\documentclass{amsart}
\usepackage[utf8]{inputenc}

\usepackage{graphics}
\usepackage{thmtools}
\usepackage{wasysym}
\usepackage[T1]{fontenc}    

\usepackage{amsthm}
\usepackage{amsbsy,amsmath,amssymb,amscd,amsfonts}
\usepackage[pagebackref=true]{hyperref}

\usepackage{graphicx,float,latexsym,color}
\usepackage[font={scriptsize,it}]{caption}
\usepackage{subcaption}

\usepackage{makecell}

\usepackage[dvipsnames]{xcolor}

\newtheorem{theorem}{Theorem}
\newtheorem*{theorem*}{Theorem}
\newtheorem{observation}{Observation}
\newtheorem{proposition}{Proposition}
\newtheorem{conjecture}{Conjecture}
\newtheorem{corollary}{Corollary}
\newtheorem{lemma}{Lemma}
\theoremstyle{remark}

\theoremstyle{definition}
\newtheorem{definition}{Definition}

\hypersetup{
    pdftoolbar=true,        
    pdfmenubar=true,        
    pdffitwindow=false,     
    pdfstartview={FitH},    
    colorlinks=true,       
    linkcolor=OliveGreen,          
    citecolor=blue,        
    filecolor=black,      
    urlcolor=red           
}

\usepackage{lineno}

\usepackage{comment}
\usepackage{microtype}
\usepackage{footnote}

\newcommand{\E}{\mathcal{E}}
\newcommand{\T}{\mathcal{T}}
\newcommand{\C}{\mathcal{C}}

\renewcommand{\P}{\mathcal{P}}
\newcommand{\D}{\mathbb{D}}
\renewcommand{\T}{\mathbb{T}}
\newcommand{\R}{\mathbb{R}}
\newcommand{\Cp}{\mathbb{C}}
\newcommand{\ol}{\overline}
\renewcommand{\l}{\lambda}
\newcommand{\X}{\mathcal{X}}

\newcommand{\ab}{_{\alpha,\beta}}

\title[Poncelet Center Power and Loci]{Invariant Center Power and Elliptic Loci\\of Poncelet Triangles}
\author[Helman, Laurain, Garcia, and Reznik]{Mark Helman, Dominique Laurain\\Ronaldo Garcia, and Dan Reznik}

\date{January, 2021}

\begin{document}

\maketitle

\begin{abstract}
We study center power with respect to circles derived from Poncelet 3-periodics (triangles) in a generic pair of ellipses as well as loci of their triangle centers. We show that (i) for any concentric pair, the power of the center with respect to either circumcircle or Euler's circle is invariant, and (ii) if a triangle center of a 3-periodic in a generic nested pair is a fixed linear combination of barycenter and circumcenter, its locus over the family is an ellipse.
\end{abstract}

\section{Introduction}
\label{sec:intro}
Poncelet N-periodics are families of N-gons inscribed in a first conic while simultaneously circumscribing a second conic \cite{dragovic11}. We continue our study of loci and invariants of Poncelet 3-periodics (see related work below). Previously we focused on families interscribed between concentric, axis-aligned ellipse pairs. Here we expand the analysis to a generic pair of nested ellipses and explore (i) the power of the center with respect to well-known circles, and (ii) loci of triangle centers under various ellipse arrangements, see Figure~\ref{fig:n3-general-pos}. Recall triangle centers are points in the plane of a triangle (e.g, incenter, circumcenter, etc.) whose trilinear coordinates obey certain conditions \cite{kimberling1993_rocky}.

\begin{figure}
    \centering
    \includegraphics[width=.6\textwidth]{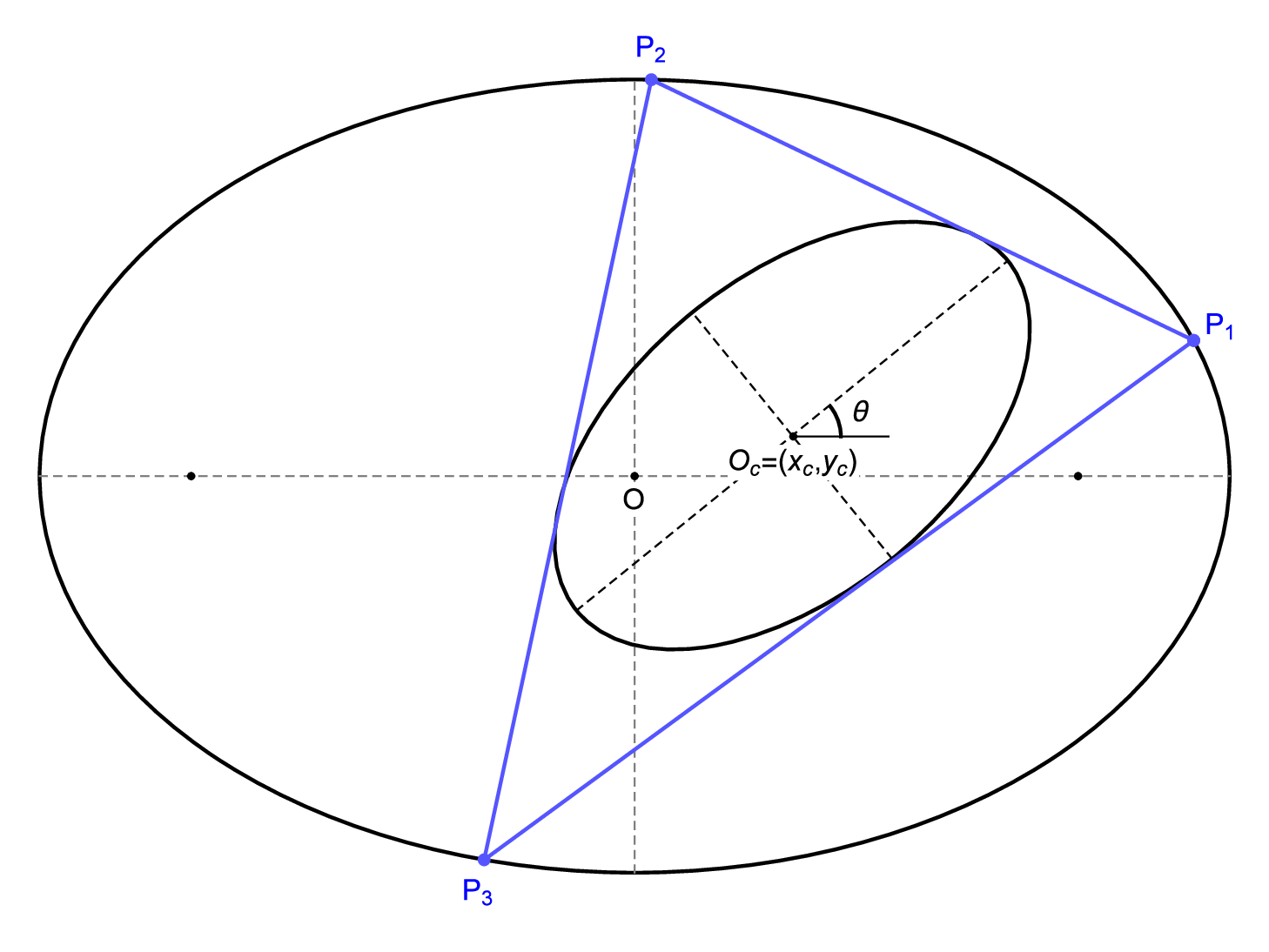}
    \caption{A pair of ellipses in general position which admits a Poncelet 3-periodic family (blue). Let the outer one be centered at the origin $O$. Their major axes are tilted by $\theta$, and their centers displaced by $O_c=(x_c,y_c)$. \href{https://youtu.be/bjHpXVyXXVc}{Video}}
    \label{fig:n3-general-pos}
\end{figure}

\subsection*{Main Results}

\begin{itemize}
  \item Section~\ref{sec:axis-aligned}: We first show that over 3-periodics in several concentric, axis-aligned pairs -- confocal, homothetic, with incircle, with circumcircle, and excentral -- the power of the center with respect to either the circumcircle or Euler's circle is invariant. Using analytic geometry (with trilinear coordinates), we derive explicit formulas for said powers for each family.
  \item Section~\ref{sec:concentric-tilted}: using CAS-based manipulation, we generalize this by proving that the power of the center with respect to both Euler's circle and the circumcircle is invariant for 3-periodics for any generic concentric pair (aligned or not), Theorem~\ref{thm:power-concentric-unaligned}.
    \item Section~\ref{sec:nonconcentric-circumcircle}: We then consider 3-periodics in the non-concentric pair with circumcircle. Using  a special parametrization based on Blaschke products \cite{daepp-2019}, we show that loci of triangle centers which are fixed affine combinations of the barycenter and circumcenter are circles, whose centers are collinear along a line passing through the stationary circumcenter.
    \item Section~\ref{sec:nonconcentric-tilted} For the generic case of 3-periodics in the non-concentric, non-axis-aligned ellipse pair, we show that triangle centers which are fixed linear combinations of barycenter and circumcenter will trace out elliptic loci, Theorem~\ref{thm:ellipse-locus}. These include such centers as the orthocenter, the center of the Euler circle, the de Longchamps point, etc.; see Observation~\ref{obs:affine-euler-line}.
\end{itemize}

In Section~\ref{sec:open-questions} we conclude with a few experimental conjectures. Appendix~\ref{app:symbols} contains a list of symbols used herein.

\subsection*{Related Work}

In \cite{odehnal2011-poristic}, the loci of many triangle centers over the poristic family (fixed circumcenter and incenter) are shown to be either stationary, circular, or elliptic. In \cite{sergei2016-com}, the loci of vertex and area centroids are proved to be ellipses over a generic Poncelet family. The circumcenter of mass (which is simply the circumcenter for Poncelet 3-periodics) is shown to be an ellipse in \cite{sergei2014-circumcenter-of-mass}.
 
 Properties of 3-periodics in the confocal pair (elliptic billiard) were studied in \cite{reznik2020-intelligencer,garcia2020-new-properties}. A few results and their subsequent proofs include: the elliptic locus of the incenter \cite{olga14,garcia2019-incenter}, circumcenter \cite{garcia2019-incenter,corentin19}, invariant sum of cosines  \cite{akopyan2020-invariants,bialy2020-invariants}, and invariant ratio of outer-to-orbit polygon areas  \cite{caliz2020-area-product}. 
 
In \cite{garcia2020-ellipses} it was shown that over confocal 3-periodics, 29 triangle centers (out of the first 100 in \cite{etc}) trace out ellipses. Explicit expressions are given for the semi-axes of each locus. In subsequent works, we studied the relationship between (i) poristic triangles and the confocal family (poristic) \cite{garcia2020-similarity-I}, and (ii) the homothetic family and the Brocard porism \cite{reznik2020-similarityII}, showing that said pairs are images of each other under a variable similarity transform. In \cite{garcia2020-family-ties} we compare several loci and invariants across several concentric, axis-aligned pairs, grouping them into clusters.

\section{Preliminaries}
\label{sec:preliminaries}


Consider two nested ellipses $\E$ and $\E_c$ with semi-axes $(a,b)$ and $(a_c,b_c)$: $\E$ is centered at the origin $O$ and $\E_c$ at $O_c=(x_c,y_c)$. Let $\theta$ denote the angle between their major axes; see Figure~\ref{fig:n3-general-pos}. If $O_c=0$ we call the pair ``concentric''. If $\theta=0$ we call it ``axis-aligned''. Additionally, let $c^2={a^2-b^2}$ and $c_c^2={a_c^2-b_c^2}$ denote their half focal distances. Note these are the squares of half the focal distance.

\begin{definition}
The power $\P_X$ of a point $X$ with respect to a circle centered on $C$ and of radius $R$ is given by \cite[Circle Power]{mw}:

\[ \P_X(C,R) = |X-C|^2-R^2 \]
\end{definition}

Recall the circumcircle passes through triangle vertices. Using Kimberling's notation for triangle centers \cite{etc}, let $X_3$ and $R$ denote its center and radius; these are known as circumcenter and circumradius. Also recall Euler's (or Feuerbach's, or the 9-point) circle: it passes through the sides' midpoints. Its radius is half the circumradius \cite[Nine-point circle]{mw}. Let $X_5$ denote its center. The following shorthands will be used for the power of a point $O$ wrt to either circumcircle or Euler's circle:

\[ \P_3 = \P_{O}(X_3,R),\;\;\;\P_5 = \P_{O}(X_5,R/2) \]

\section{Concentric, Axis-Aligned: Invariant Power of Origin}
\label{sec:axis-aligned}
In this Section we study the power of the center of the system over several classic concentric, axis-aligned Poncelet families. A key observation is that said power remains constant with respect to two classic circles (the circumcircle and Euler's circle), despite the fact that their centers and radii variable.

\subsection{Pair with Incircle}

Consider an ellipse pair where the inner ellipse is a circle of radius $r$; see Figure~\ref{fig:incircle}. The Cayley condition for 3-periodics reduces to $r=(a b)/(a+b)$ \cite[Coroll. 1]{garcia2020-family-ties}. By definition, the incenter $X_1$ lies at the origin and the inradius is constant.

Remarkable properties of this family include the fact that (i) it conserves the sum of cosines, (ii) the circumradius $R$ is invariant, and (iii) the locus of both $X_3$ and $X_5$ are circles\footnote{Amongst the first 201 centers in \cite{etc} the loci of the following are circles concentric with $X_1$: $X_k$, $k=$3, 5, 11, 12, 35, 36, 40, 46, 55, 56, 57, 65, 80, 119, 165 \cite{garcia2020-family-ties}.} concentric with $X_1$. Let $r_3$ and $r_5$ denote their radii, respectively. These are given by \cite[Section 3]{garcia2020-family-ties}:

\[ R = \frac{a+b}{2},\;\;\;r_3 = \frac{a-b}{2},\;\;\;r_5 = \frac{(a-b)^2}{4(a+b)}\]

\begin{proposition}
Over 3-periodics in the concentric pair with incircle, the power of the center $O=X_1$ with respect to either circumcircle \cite[Prop. 1]{garcia2020-family-ties} and Euler's circle are invariant and given by:

\[ \P_3= -a b,\;\;\;\P_5= -a b \frac{a^2 + b^2}{2(a + b)^2}\]

\end{proposition}

\begin{proof}
We compute powers wrt circumcircle and Euler circle as functions of sidelengths $s_1,s_2,s_3$:

\[ \P_3 = -\frac{s_1 s_2 s_3}{s_1 + s_2 + s_3} \]

\[\P_5 = \frac{-(s_1^3 + s_2^3 + s_3^3 - s_1^2(s_2 + s_3) - s_2^2(s_3 + s_1) - s_3^2(s_1 + s_2) + 4 s_1 s_2 s_3 )}{4 (s_1 + s_2 + s_3)} \]

$\P_3$ is $-1/{\pi}$ times the area of the circumellipse centered on the incenter, i.e., $\P_3 = -a b$. In order to simplify $\P_5$ formula we derive the following parametrization for the sidelengths:

\[ s_1 = \frac{w \tau}{r},\;\;\;
 s_2 = \frac{4r(r^2 + (1 - t)w)}{(2r^2 + (1 - t)w)\tau + (1 - t)z},\;\;\;
 s_3 = \frac{4r(r^2 + (1 - t)w)}{(2r^2 + (1 - t)w)\tau - (1 - t)z}\]

\noindent with $\tau=\sqrt{1-t^2}$ and $z =\sqrt{4r^2(t - 1)w + (1 - t^2)w^2 - 4r^4}$, $r =(ab)/(a + b)$, and $w = ab$. Replacing $s_1$, $s_2$, $s_3$ obtain $\P_5 = r^2 - 1/2 w$ and the formula in the proposition.
\end{proof}

\begin{figure}
    \centering
    \includegraphics[width=.7\textwidth]{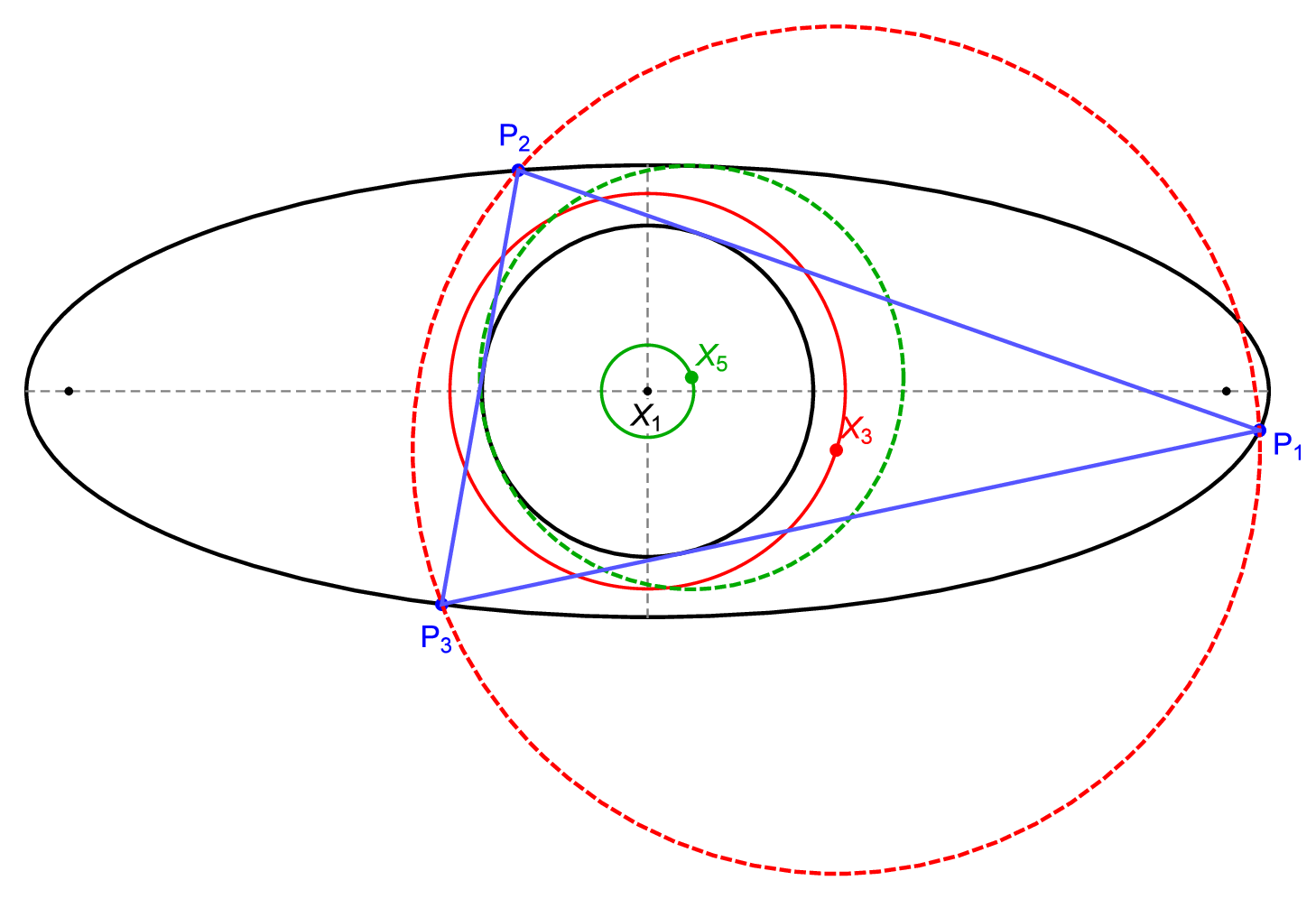}
    \caption{In the ellipse pair with incircle, the locus of both $X_3$ and $X_5$ are concentric circles (red and green). The power of the center $O=X_1$ wrt to either the circumcircle (dashed red) or Euler's circle (dashed green) is invariant. \href{https://bit.ly/3snkrJc}{live}}
    \label{fig:incircle}
\end{figure}

\subsection{Pair with Circumcircle}

Consider an ellipse pair where the outer ellipse is a circle of radius $a=b=R$ and the inner one is a concentric ellipse  with semi-axes $(a_c,b_c)$. The Cayley condition for 3-periodics to exist reduces to $a_c+b_c=R$ \cite{garcia2020-family-ties}. By definition, the incenter $X_3$ lies at the origin and the circumradius $R$ is constant. Invariants known to this family include the product of cosines and the sum of sidelengths squared \cite{garcia2020-family-ties}. Referring to Figure~\ref{fig:circumcircle}:

\begin{proposition}
Over 3-periodics in the concentric pair with circumcircle, the locus of $X_5$ is a circle\footnote{Amongst the first 201 centers in \cite{etc} the loci of the following are circles concentric with $X_3$: $X_k$, $k=$2, 4, 5, 20, 22, 23, 24, 25, 26, 74, 98, 99, 100, 101, 102, 103, 104, 105, 106, 107, 108, 109, 110, 111, 112, 140, 156, 186, 201 \cite{garcia2020-family-ties}.} concentric with original pair whose radius $r_5$ is given by $r_5= ({a_c-b_c})/{2}$. 
\end{proposition}

\begin{figure}
    \centering
    \includegraphics[width=.5\textwidth]{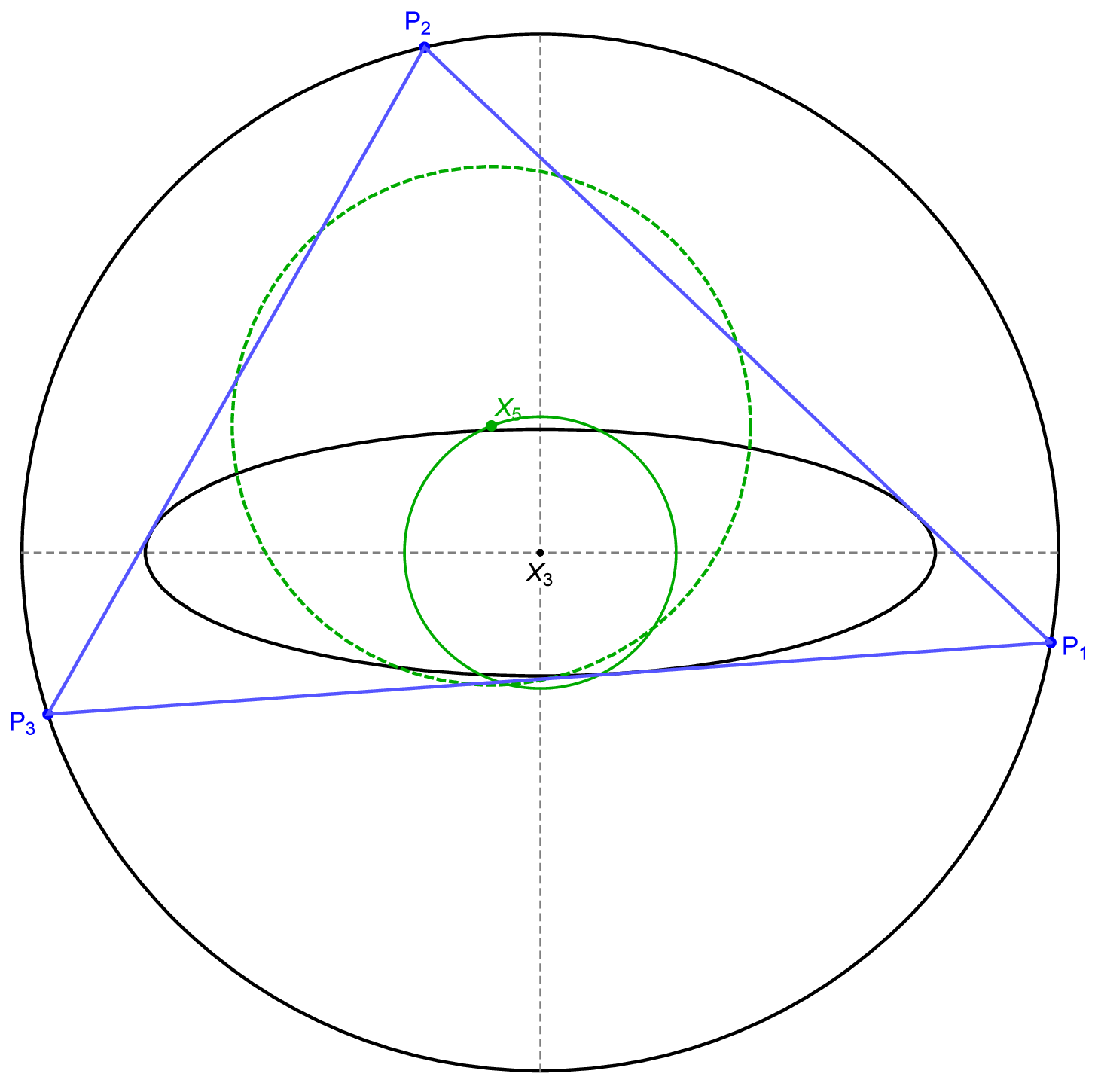}
    \caption{A circle and a concentric inellipse and a  3-periodic (blue). By definition, $X_3$ is stationary at the common center and the circumradius $R$ is constant. The locus of the center $X_5$ of Euler's circle (dashed green) is a concentric circle (solid green). \href{https://bit.ly/31nIg84}{live}}
    \label{fig:circumcircle}
\end{figure}

\begin{corollary}
Over 3-periodics in the concentric pair with circumcircle, the power of the center $O=X_3$ with respect to either circumcircle or Euler's circle is invariant and given by:

\[ \P_3 = -R^2,\;\;\;\P_5 = r_5^2-(R/2)^2 = -a_c b_c\]
\end{corollary}

\subsection{Homothetic Pair}

Consider a pair of concentric, homothetic ellipses admitting a 3-periodic family (elliptic billiard); see Figure~\ref{fig:homothetic}.

\begin{figure}
    \centering
    \includegraphics[width=.7\textwidth]{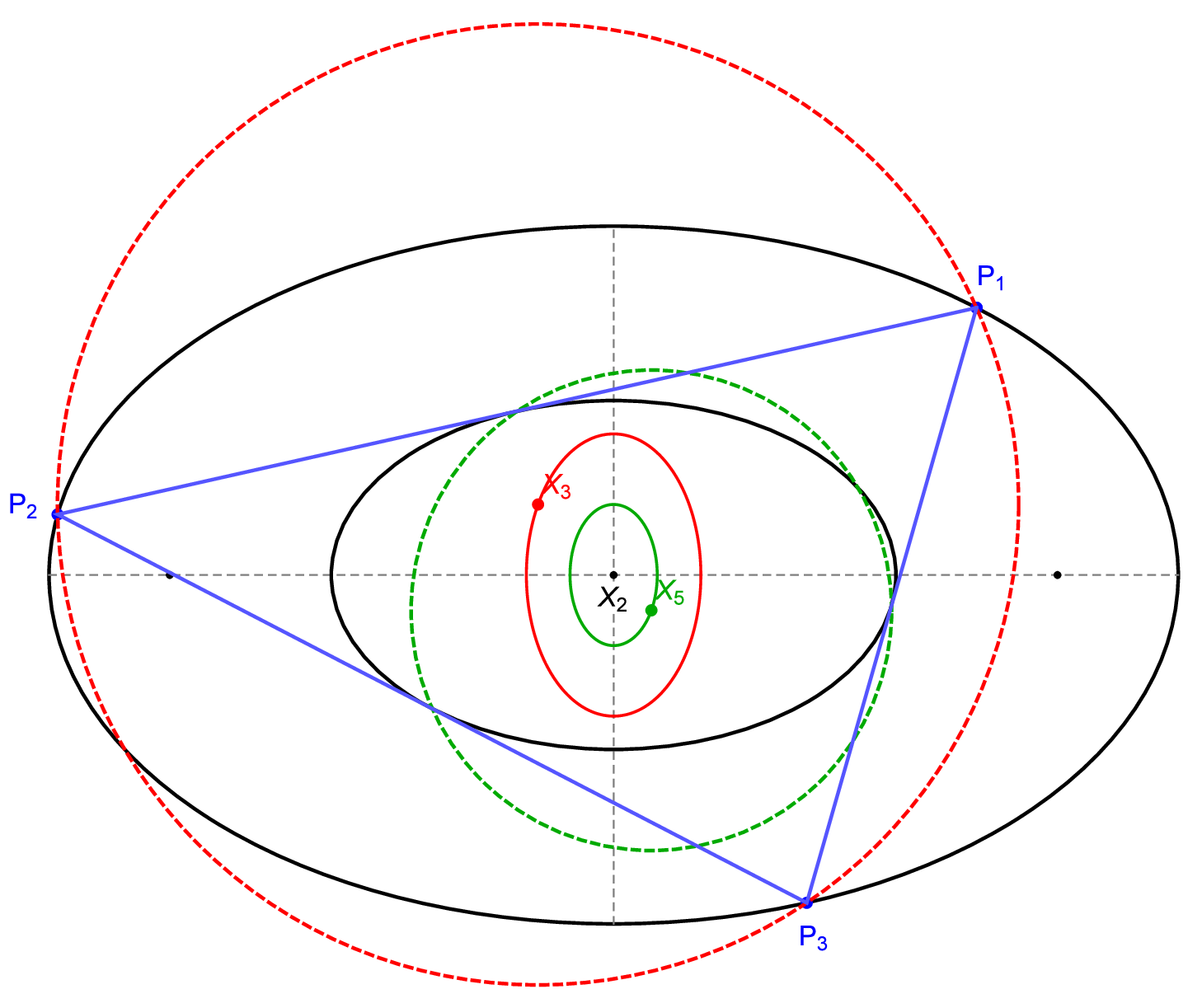}
    \caption{A concentric, homothetic pair of ellipses (black) and a 3-periodic (blue) interscribed between them. The barycenter $X_2$ is stationary at the common center. Over the family, the locus of $X_3$ and $X_5$ are concentric ellipses (solid red and green). The power of the center wrt the circumcircle (dashed red) and Euler's circle (dashed green) is invariant. \href{https://bit.ly/3vZEr6W}{live}}
    \label{fig:homothetic}
\end{figure}

This family conserves area and sum of squared sidelengths, and consequently the Brocard angle \cite{reznik2020-similarityII}. The Barycenter $X_2$ is stationary at $O$. The Cayley condition implies that $(a_c,b_c)=(a/2,b/2)$.

Over this family, the locus of both $X_3$ and $X_5$ are ellipses concentric and axis-aligned with the original pair \cite{garcia2020-family-ties}. Their semi-axes are given by:

\[(a_3,b_3)=\frac{c^2}{4}\left(\frac{1}{a},\frac{1}{b}\right),\;\;\text{and}\;\;
	(a_5,b_5)=\frac{c^2}{8}\left(\frac{1}{a},\frac{1}{b}\right) \]


\noindent Nevertheless:

\begin{proposition}
Over 3-periodics in the homothetic pair, the power of the center $O=X_2$ with respect to both the circumcircle and Euler circle are invariant and given by:

\[ 
\P_3 = -\frac{a^2+b^2}{2},\;\;\;\P_5=-\frac{a^2+b^2}{8} \]
\end{proposition} 

\begin{proof}
Squared radii and squared distances between barycenter $X_2$ and circumcenter $X_3$ or nine-point center $X_5$ can be obtained via direct computation and in terms of squared sidelengths. This yields:

\[
\P_3 = -\frac{1}{9}\ \sum{ s_i^2},\;\;\;
\P_5= -\frac{1}{36}\sum{s_i^2} \]

Recall the sum of squared sidelengths is conserved in the homothetic pair and given by \cite[Remark 2.1]{reznik2020-similarityII}:

\[  \sum{s_i}^2=\frac{9}{2}(a^2+b^2)\]

\end{proof}

\subsection{Confocal Pair}

Consider a confocal pair of ellipses which admits a 3-periodic family (elliptic billiard); see Figure~\ref{fig:confocal}. Classic invariants include perimeter and Joachmisthal's constant \cite{sergei91}. A recent result is that the sum of angle cosines is invariant \cite{reznik2020-intelligencer,garcia2020-new-properties}. The Mittenpunkt $X_9$ is stationary at $O$ \cite{reznik2020-intelligencer}. The semi-axes of the inner ellipse are given by \cite{garcia2019-incenter}:

\begin{equation}
[a_c,b_c] = \frac{1}{c^2}\left[a(\delta-b^2),\;b(a^2-\delta)\right]
\label{eqn:confocal-axes}
\end{equation} 

\noindent where $\delta^2={a^4-a^2 b^2+b^4}$.

Furthermore, over said family, the locus of both $X_3$ and $X_5$ are concentric, axis-aligned ellipses with semi-axes are given by \cite{garcia2020-ellipses}:

\begin{align*}
    [a_3,b_3]=&\left[\frac{a^2-\delta}{2a},\frac{\delta-b^2}{2b}\right]\\
    [a_5,b_5]=&\left[\frac{- w_2(a,b)+ w_3(a,b) \delta}{ w_1(a,b)},\;\frac{ w_2(b,a)-{w_3(b,a) \delta}}{w_1(b,a)}\right]
\end{align*}

\noindent where $w_1(u,v)=4u(u^2-v^2)$, $w_2(u,v)=u^2(u^2+3v^2)$, $w_3(u,v)=3u^2+ v^2$.

Recall that over 3-periodics in the confocal pair, $\P_3=-\delta$ \cite[Thm. 3]{garcia2020-new-properties}. We extend this to $\P_5$:

\begin{proposition}
 Over 3-periodics in the confocal pair, the power of the center $O=X_9$ with respect to Euler's circle is invariant and given by:




\[ \P_5 =  \frac{\delta \mu \eta (\mu^2 + \eta^2 - 2 )}{\mu^2 + \eta^2 + 1} \]

\noindent where $\mu= {a}/{a_c}$ and $\eta= {b}/{b_c}$.
\end{proposition}

\begin{proof}

Straightforward computation of power of $O=X_9$ with respect to circumcircle and Euler circle gives:

\[P_3 = \frac{ -(s_1^3 + s_2^3 + s_3^3 - s_1^2 (s_2 + s_3) - s_2^2 (s_3 + s_1)  - s_3^2 (s_2 + s_1) + 6 s_1 s_2 s_3) s_1 s_2 s_3 }{(s_1^2 + s_2^2 + s_3^2 - 2 s_1 s_2 - 2 s_2 s_3 - 2 s_3 s_1)^2 }\]

\[\P_5 = \kappa \frac{ (s_1^3 + s_2^3 + s_3^3 - s_1^2 (s_2 + s_3) - s_2^2 (s_3 + s_1)  - s_3^2 (s_2 + s_1))}{(s_1^2 + s_2^2 + s_3^2 - 2 s_1 s_2 - 2 s_2 s_3 - 2 s_3 s_1)^2 } \]

\noindent where $\kappa=(s_1 + s_2 - s_3)(s_1 - s_2 + s_3)(-s_1 + s_2 + s_3)$.

We use the following parametrization for 3-periodics $P_1P_2P_3$ in the confocal pair. Let $\rho$ denote the invariant $r/R$ ratio (inradius/circumradius), $s$ the invariant semi-perimeter, and $c_1$ the cosine of the internal angle at $P_1$:

\[ s_1 = \frac{2(1 - c_1) s}{\rho - 2 c_1 + 2},\;\;\;s_2 = \frac{(\rho c_1 - c_1^2 + \rho + 1 - w) s}{(1 + c_1)(\rho - 2 c_1 + 2)},\;\;\;s_3 = \frac{(\rho c_1 - c_1^2 + \rho + 1 + w) s}{(1 + c_1)(\rho - 2 c_1 + 2)}\]

\noindent where $h^2=1-2\rho$, and $w^2 = (1 - c_1^2) (h^2 - (c_1 - \rho)^2)$. The squared semi-axis lengths are given by: 

\[ a^2 = \frac{4(1 + h)s^2}{(3 - h)(3 + h)^2},\;\;\;b^2 = \frac{4(1 - h)s^2}{(3 + h)(3 - h)^2} \]

\noindent Substituting $s_1$, $s_2$, $s_3$ in $\P_3$ and $\P_5$ we get powers with respect to the circles:

\[ \P_3 = \frac{-4(3 + h^2) s^2}{(9 - h^2)^2},\;\;\;
\P_5 = \frac{-(3 - h^2)(1 - h^2)s^2}{(9 - h^2)^2} \]

which doesn't depend of variable parameter $c_1$.

These values are always negative since $O$ is interior to both the circumcircle and Euler circle.

Finally, using $a_c/a = (1 + h)/2$ and $b_c/b = (1 - h)/2$ we get the formulas in the proposition.


\end{proof}

\begin{figure}
    \centering
    \includegraphics[width=.7\textwidth]{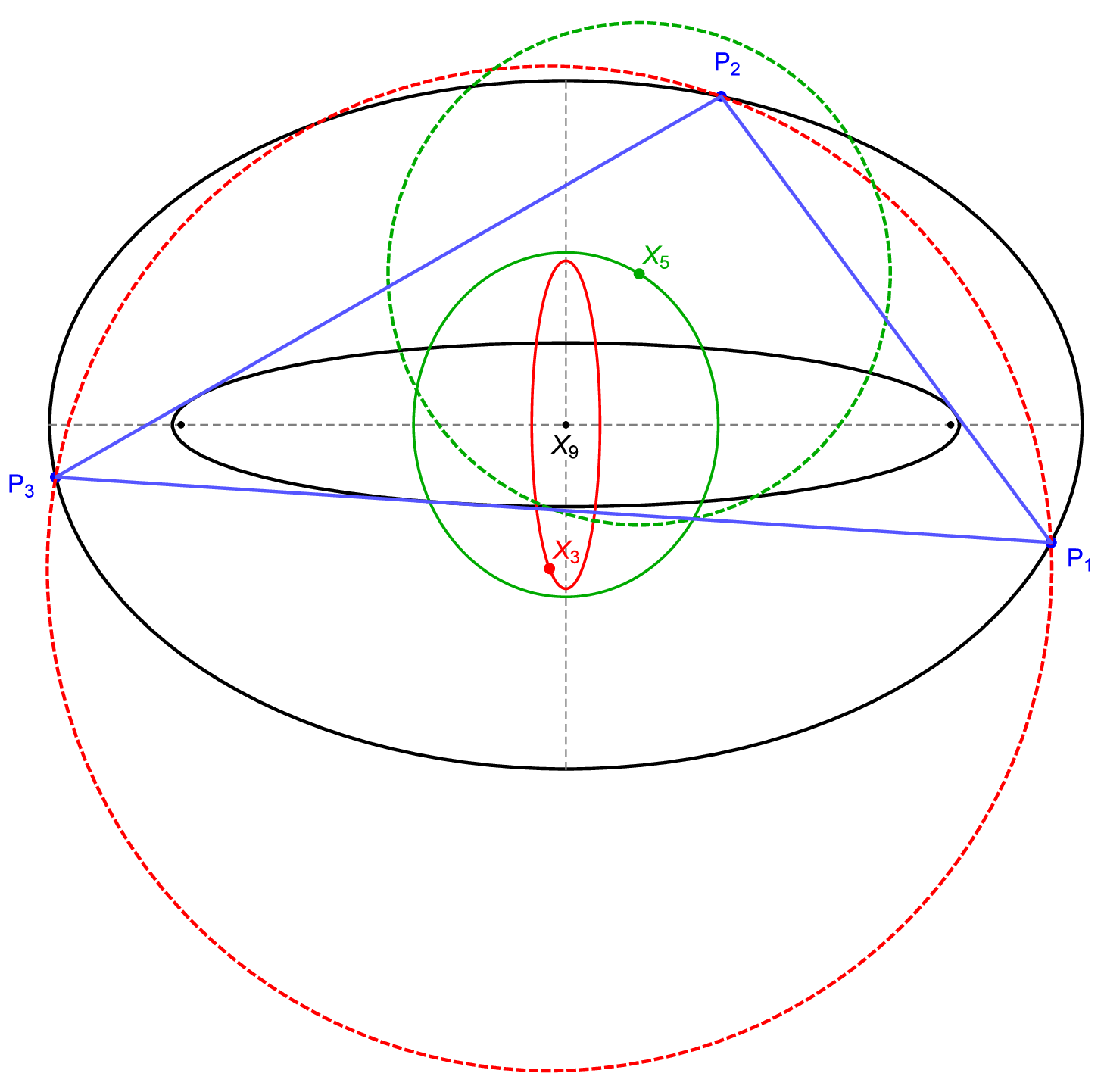}
    \caption{In the confocal pair (elliptic billiard), the locus of both $X_3$ and $X_5$ are concentric ellipses, (red and green), axis aligned with the original ones. The power of the center $O=X_9$ wrt to either the circumcircle (dashed red) or Euler's circle (dashed green) is invariant. \href{https://bit.ly/3w151N5}{live}}
    \label{fig:confocal}
\end{figure}

\subsection{Excentral to Confocals}

Referring to Figure~\ref{fig:excentral}, the excentral family comprises the excentral triangles to 3-periodics in the elliptic billiard. Abusing notation, here we let $a,b$ denote the axes of said elliptic billiard (i.e., the caustic to the excentral family), and $a_e,b_e$ denote the axes of outer ellipse $\E$, which in \cite{garcia2019-incenter} was derived as:

\begin{equation}
a_e = (a^2+\delta)/b, \;\;\;b_e = (b^2+\delta)/a
\label{eqn:excentrals}
\end{equation}

\noindent where $\delta$ is as in \eqref{eqn:confocal-axes}. Since the Euler's circle of an excentral triangle is the circumcircle of the reference:

\begin{observation}
Over the excentral family $\P_5=-\delta$.
\end{observation}

Still referring to Figure~\ref{fig:excentral}:

\begin{proposition}
In the family of excentral triangles to the confocal family, $\P_3$ is invariant and given by:

\[ \P_3 = -a^2-b^2-2\delta \]
\end{proposition}

\begin{proof}

In the billiard family, straightforward computation of power of $O=X_9$ with respect to Bevan circle (circumcircle of excentral triangle) gives:

\[ \P_3 = \frac{ -12 (s_1 s_2 s_3)^2 }{ (s_1^2 + s_2^2 + s_3^2 - 2 s_1 s_2 - 2 s_2 s_3 - 2 s_3 s_1)^2 }\]

or

\[ \P_3 = \frac{ -48 s^2 }{ (9 - h^2)^2 } =  -a^2-b^2-2\delta \]

\end{proof}

\begin{figure}
    \centering
    \includegraphics[width=.8\textwidth]{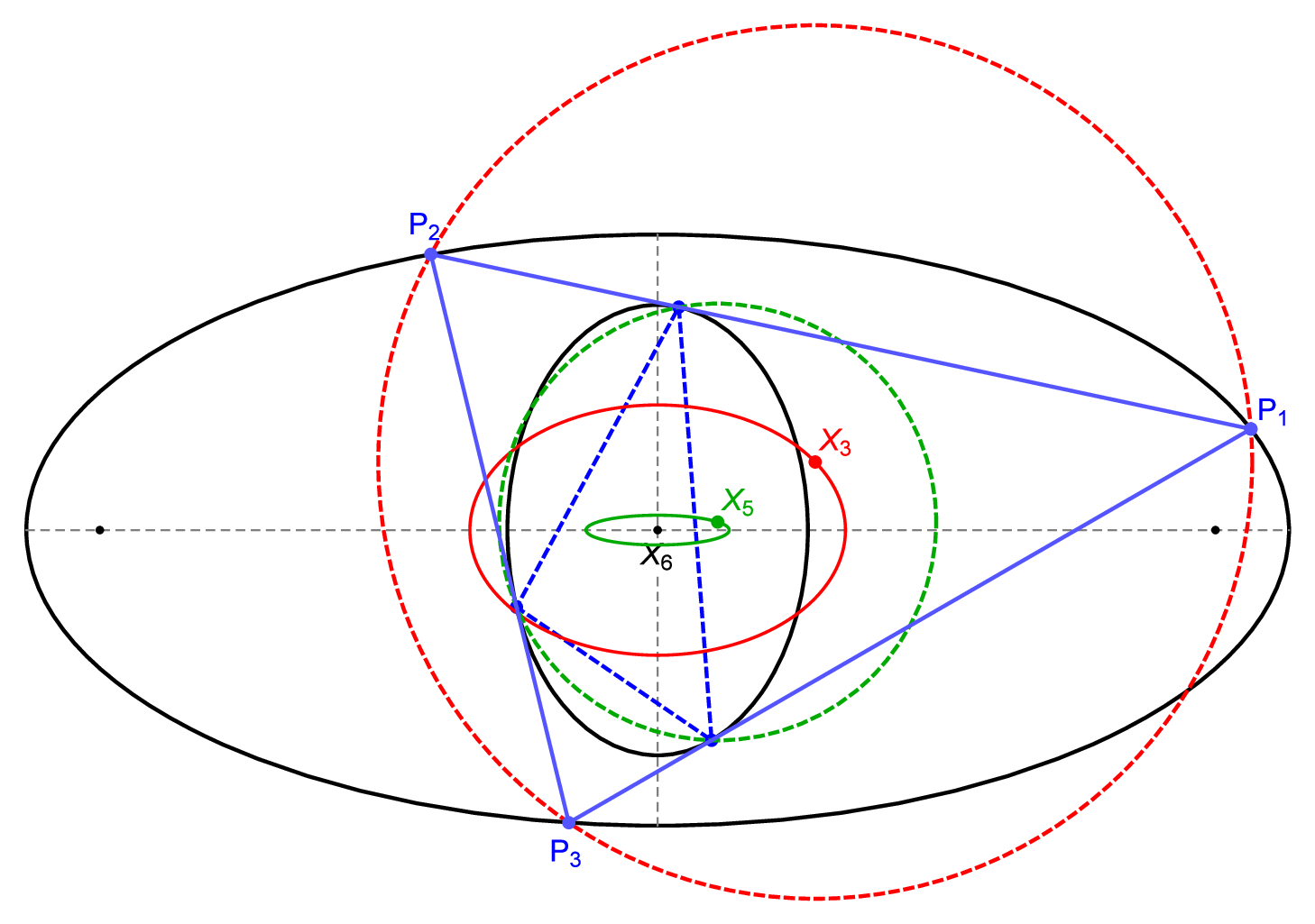}
    \caption{The excentral family (solid blue) has an elliptic billiard as its caustic, and billiard 3-periodics (dashed blue) as orthic triangles. The power of the symmedian point $X_6$, stationary at the center, with respect to circumcircle (dashed red) and Euler circle (dashed green) is invariant. Also shown are the elliptic loci of their centers $X_3$ and $X_5$ (solid red, solid green), respectively. \href{https://bit.ly/3tV16iY}{live}}
    \label{fig:excentral}
\end{figure}

Note that $\P_3$ for the excentral family is equivalent to the power of the power of the center with respect to the Bevan circle in the confocal family (circumcircle of excenters \cite[Bevan Circle]{mw}).

\subsection*{Summary}

Table~\ref{tab:concentric-summary} summarizes some results in this section. As it will be seen in the next Section, these are special cases of Theorem~\ref{thm:power-concentric-unaligned}.

\begin{table}
\centering
\begin{tabular}{|r|c|l|c|c|}
\hline
family & $N=3$ Cayley & $\P_3$ & $\P_5$ \\
\hline
incircle & $r=(ab)/(a+b)$ & $-ab$ &  $-a b (a^2 + b^2)/(2(a + b)^2)$  \\
circumc. & $a_c+b_c=R$ & $-R^2$ & $-a_c b_c$ \\
homoth. & $a_c=a/2,b_c=b/2$  & $-({a^2+b^2})/{2}$ & $-({a^2+b^2})/{8}$  \\
confocal & see \eqref{eqn:confocal-axes} & $-\delta$ & ${\delta \mu \eta (\mu^2+\eta^2-2)}/{(\mu^2+\eta^2+1)}$\\
excentral & see \eqref{eqn:excentrals} & $-a^2-b^2-2\delta$ & $-\delta$ \\
\hline
\end{tabular}
\caption{Summary of invariant power of origin wrt to circumcircle and Euler's circle for various concentric, axis-aligned systems. Recall $\mu=a/a_c$ and $\eta=b/b_c$.}
\label{tab:concentric-summary}
\end{table}


In Appendix~\ref{app:four-more} we show several additional circle-family combinations over which the power of the center is invariant.

\section{Concentric and Unaligned}
\label{sec:concentric-tilted}
We generalize the previous results by showing that over Poncelet 3-periodics in concentric, non-axis-aligned ellipse pairs, the power of the center with respect to either circumcircle or Euler's circle
is still invariant. This holds despite (i) their centers moving along non-axis aligned ellipses and (ii) their radii being variable.

The Cayley condition for the concentric version ($O_c=O$) of the pair in Figure~\ref{fig:n3-general-pos} which admits a 3-periodic family reduces to \cite{dragovic11}: 


\[ a^2 b^2+\cos^2{\theta}~c^2 (a_c^2-b_c^2)-(a a_c+b b_c)^2 = 0\]

Note $\cos^2{\theta}$ can be expressed in terms of $a,b,a_c,b_c$:


\[ \cos^2{\theta} =  \frac{(a a_c+b b_c)^2-a^2 b^2}{c^2 (a_c^2-b_c^2)} \]

As shown ion Figure~\ref{fig:cos-limits}, the feasible region of $a_c,b_c$ lies between to lines. Note that when $\theta=0$ the above reduces to:

\[ \frac{a_c}{a}+\frac{b_c}{b}=1 \]

\begin{figure}
    \centering
    \includegraphics[width=.5\textwidth]{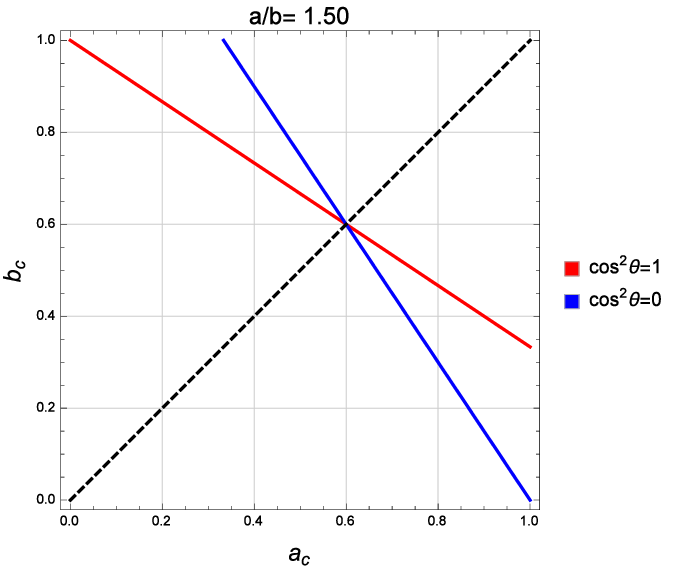}
    \caption{For a given choice of $a$ and $b$ (in the figure, $a=1.5$, $b=1$, the feasible region for $a_c$ and $b_c$ lies between the limit lines where $\cos^2(\theta)$ is 0 or 1, respectively.}
    \label{fig:cos-limits}
\end{figure}

Referring to Figure~\ref{fig:n3-tilted-circum-euler}, consider the family of 3-periodics interscribed between a concentric pair of ellipses $\E$ and $\E_c$, at an angle $\theta$ with each other.
Consider the pair of ellipses:

\begin{align*}
   \E&:  (b^2 + c^2)x^2  - 2 a c x y  + a^2 y^2 - b^2 a^2=0 \\
   \E_c&: \frac{x^2}{a_c^2}+\frac{y^2}{b_c^2}-1=0\\
\end{align*} 
where $-c^2a_c^2 + (ab + a b_c + a_c b)(ab - ab_c - a_c b)=0$.

\begin{proposition}
The locus of $X_3$ and $X_5$ are ellipses $\E_3$, and $\E_5$  which are concentric with the pair. Furthermore, $\E_3$ is axis-aligned with $\E_c$ and its aspect ratio is equal to $b_c/a_c$. $\E_3$ is given by:

\begin{align*}
%
\E_3:&\; 4\,{b}^{2}a_c^{2} \left( a_c^{2}{x}^{2}+b_c^{2}{y}^{2} \right)  
 -b_c^{2} \left(  \left( a_c^{2}-
b_c^{2} +{b}^{2}\right)^2 a^2 -4\,a_c^2b^2\, \left( a_c-b_c
 \right)^2   \right)    =0
\end{align*}
\label{prop:loci-x3-x5}

\noindent The axes of $\E_5$ are given by

\begin{align*}
    a_5&=\frac{\sqrt{2} (a a_c^2 b_c + a b^2 b_c - a b_c^3 - 2 a_c^3 b - 2 a_c b b_c^2)}{8 a_c^2 b}\\
    b_5&=\frac{\sqrt{2} (a a_c^2 - 4 a_c b b_c + a (b^2 - b_c^2))}{8 b a_c}
\end{align*} 
\end{proposition}

\begin{observation}
For the special case where the pair is axis-aligned $(\theta=0)$ the expression for $\E_5$ is tractable, and given by:

\[\E_5: \; \frac{ 16 a^4x^2}{(a^3 - 3 a^2  a_c + a b^2 -  a_c b^2)^2}+\frac{16 b^2 a^2y^2}{ (a^2  a_c - 2 a b^2 + 3  a_c b^2)^2}-1=0\]
\end{observation}

Now we are in a position to prove our first main result:

\begin{theorem}
The power of the common center $O$ is invariant with respect to either the circumcircle or the Euler circle and given by:

\begin{align*}
    \P_3  &=-\frac{a b_c}{b a_c}\left(b^2 + a_c^2 - b_c^2 \right)- (a_c^2 - b_c^2) \\
    \P_5&= -\frac{a b_c}{2 b a_c}\left(b^2 + a_c^2 - b_c^2 \right)+b_c^{2}
\end{align*}

\label{thm:power-concentric-unaligned}
\end{theorem}

\begin{proof}
We write explicit parametrized expressions for the vertices of 3-periodics in a generic concentric pair. We write out explicit expressions for circumcenter, Euler center, and circumradius and finally the power of the center with respect to these. Via a process of laborious manual CAS-based simplification, we arrive at the result.
\end{proof}

\noindent Notice that $2\P_5-\P_3=a_c^2+b_c^2$.

\begin{observation}
When the ellipses are concentric and axis-aligned, Theorem~\ref{thm:power-concentric-unaligned} reduces to:

\[ \P_3  = -\frac{a_c}{a}c^2 - b^2,\;\;\;\P_5 =-\frac{{a_c} (a-{a_c}) \left(a^2+b^2\right)}{2 a^2}\]
\end{observation}

\begin{figure}
    \centering
    \includegraphics[width=.7\textwidth]{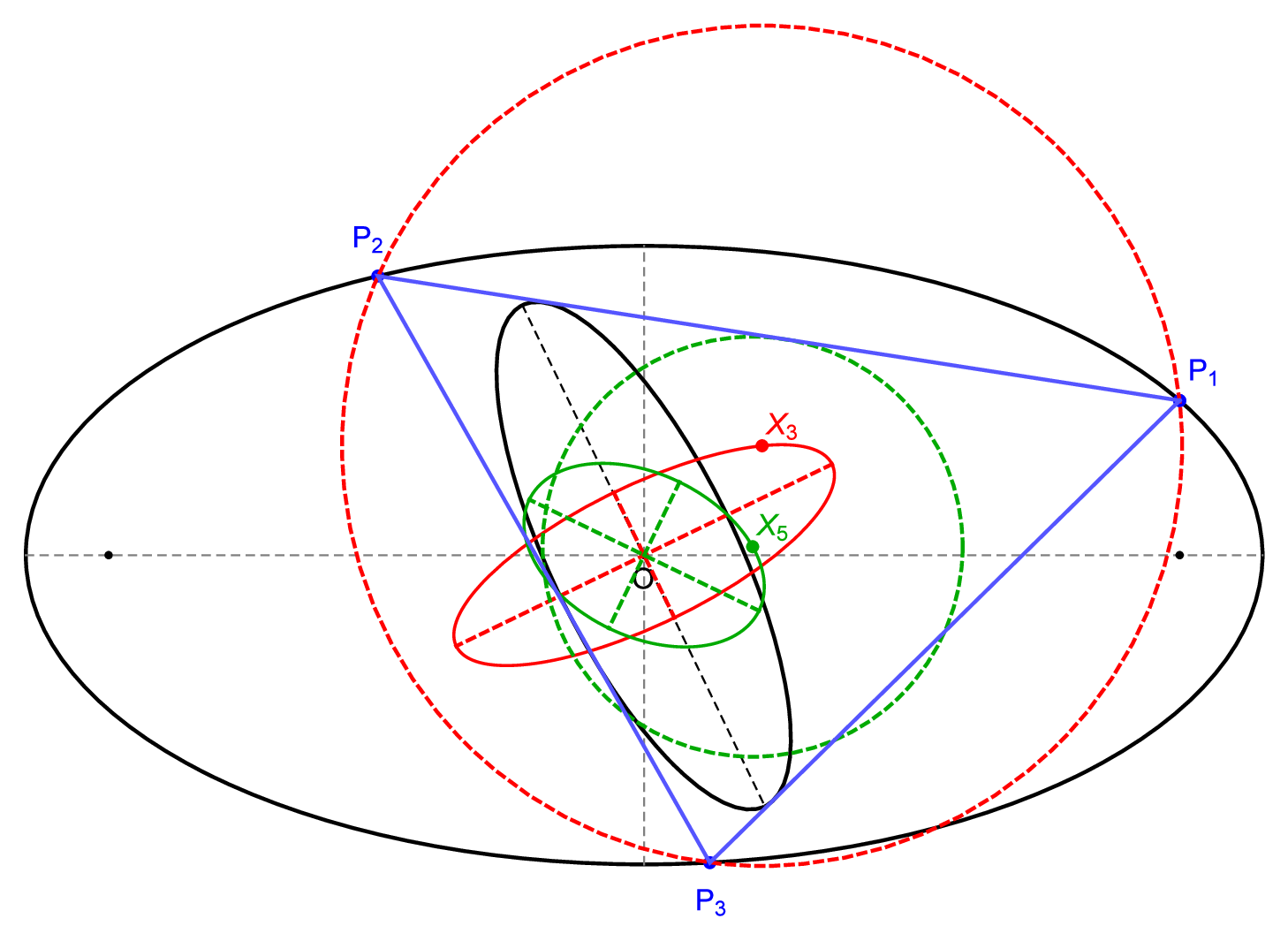}
    \caption{Over the family of 3-periodics (blue) interscribed in a pair of concentric, unaligned ellipses, the locus of $X_3$ and $X_5$ are also ellipses (red and green) which are concentric with the original pair. Furthermore, the locus of $X_3$ is axis-aligned with the inner ellipse. Remarkably, the power of the common center $O$ with respect to either the circumcircle (dashed red) or Euler's circle (dashed green) is invariant.}
    \label{fig:n3-tilted-circum-euler}
\end{figure}

Referring to Figure~\ref{fig:concentric-xns}, consider the concentric, unaligned pair of ellipses $\E$ and $\E_c$ given by:

\[ 
    \E:\,\frac{x^2}{a^2}+\frac{y^2}{b^2}-1=0,\;\;\;
    \E_c:\,(b_c^2 + \zeta^2)x^2  - 2 a_c \zeta x y  + a_c^2 y^2 - b_c^2 a_c^2=0\\
\]

where $\zeta$ is defined relative to $\theta$ as follows:

\[\tan{2\theta}=\frac{2 a\zeta}{c^2 - \zeta^2}\]

Furthermore the Cayley condition reduces to:

\[a^2 \zeta^2 - (a b + a b_c + a_c b)(a b - a b_c - a_c b)=0\]

\begin{figure}
     \includegraphics[width=\textwidth]{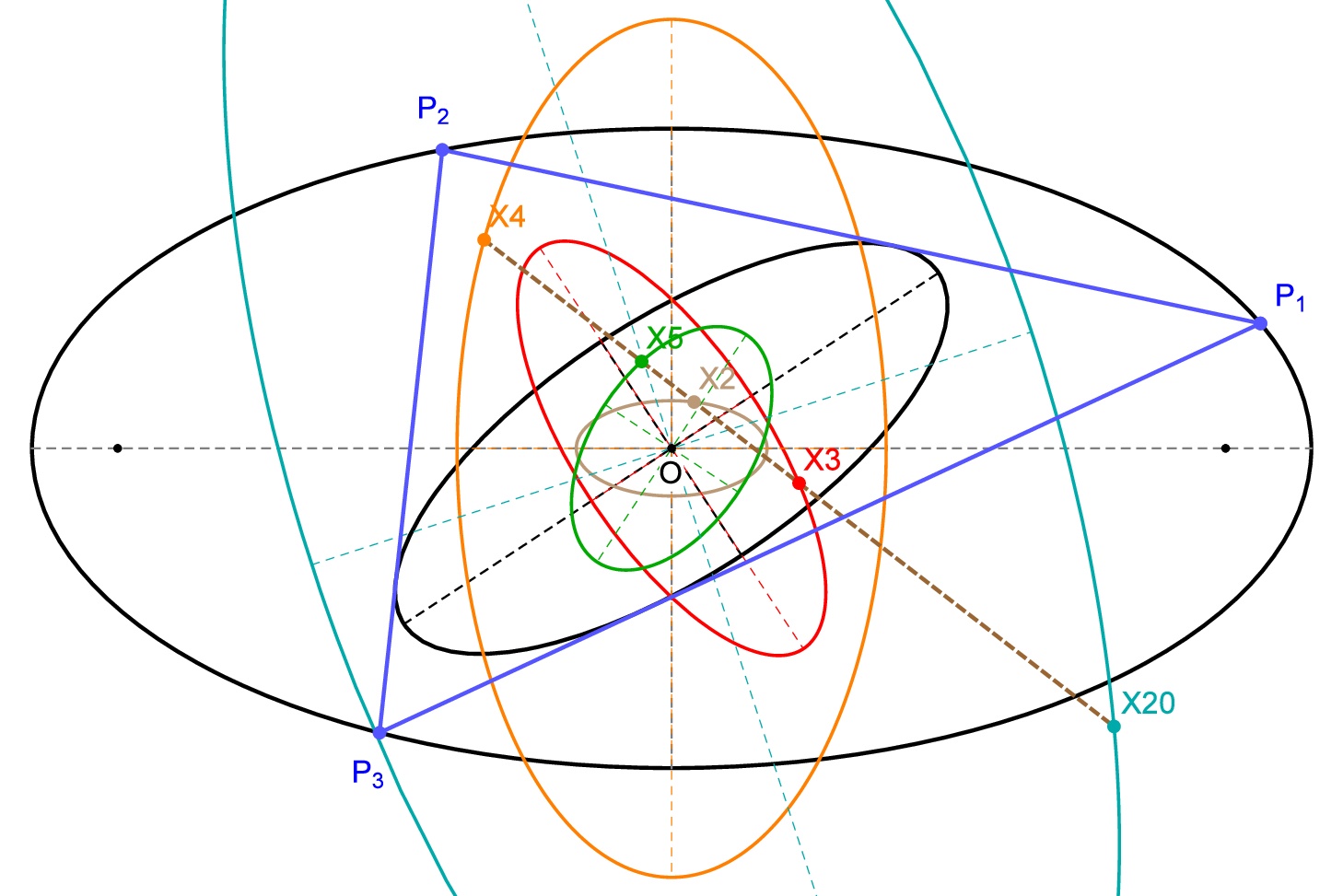}
     \caption{In the concentric-tilted pair (black ellipses), the loci of $X_k$, $k=2,3,4,5,20$ are ellipses. The Euler line (dashed brown) connect said centers. \href{https://youtu.be/hpb7ZgKWjUY}{Video}}
     \label{fig:concentric-xns}
 \end{figure}

\subsection*{Generalizing invariant center power}

Recall a pencil of coaxial circles has collinear centers and identical limiting points \cite[Coaxal Circles]{mw}.

The radical axis of the pencil of circles containing the circumcircle and Euler's circle is known as the ``orthic axis'' and it is perpendicular to the Euler line\footnote{The Euler line and the orthic axis intersect at $X_{468}$.} \cite[Orthic Axis]{mw}.

\begin{proposition}
Over Poncelet 3-periodics in a concentric pair (axis-aligned or not) with center at $O$, let $\C$ be a circle coaxial with both circumcenter and Euler's circle, such that its center $X$ is a fixed affine combination of $X_3$ and $X_5$. The power of $O$ with respect $\C$ is invariant.
\label{prop:coaxial}
\end{proposition}

\begin{proof}

Let $\P(M,\C_0)$ denote the power of point $M\in\R^2$ with respect to a circle $\C_0$.

Let $\C_3$ (resp. $\C_5$) denote the circumcircle (resp. Euler's circle) of 3-periodics in a concentric pair centered on $O$. Let $X_3$ and $R$ (resp. $X_5$ and $r_5=R/2$) denote the center and radius of $\C_3$ (resp. $\C_5$). Let $r_X$ denote the radius of $\C$, a circle coaxial with $\C_3$ and $\C_5$, centered at $X$.

Let $t$ be a fixed real number, so that $X=t X_3+(1-t)X_5$ is a fixed affine combination $X_3$ and $X_5$. By Theorem \ref{thm:power-concentric-unaligned}, $\P(O,\C_3)$ and $\P(O,\C_5)$ are both invariant over 3-periodics in a concentric pair.

In the remainder of this proof we show that in fact:

\[ \P(O,\C)=t\P(O,\C_3)+(1-t)\P(O,\C_5)\]
which entails our claim.

Take a single 3-periodic in the aforementioned family. Since both statements $X=t X_3+(1-t)X_5$ and $\P(O,\C)=t\P(O,\C_3)+(1-t)\P(O,\C_5)$ are invariant under rotations and translations of the plane, assume without loss of generality that the line through $X_3$, $X_5$, and $X$ (Euler line) is the x-axis and that the common radical axis of $\C_3$, $\C_5$, and $\C$ (orthic axis) is the y-axis. Let $T$ be the origin of this coordinate system ($X_{468}$ of the 3-periodic triangle).

Let $X_3=(x_3,0)$ and $X_5=(x_5,0)$ so that $X=(t x_3+(1-t)x_5,0)$. Also, let $O=(u,v)$ in this new coordinate system. For an arbitrary point $P=(x_P,0)$ on the line through $X_3$, $X_5$, we denote by $\C_P$ the circle centered at $P$ that is coaxial with $\C_3$ and $\C_5$, and its radius by $r_P$.

Since $T$ is on the radical axis of the pencil of coaxial circles of $\C_3$ and $\C_5$, its power with respect to any of the circles in this pencil is the same, so we denote it as $\P_T$. See that 
\begin{gather*}
    \P(O,\C_P)=|O P|^2-{r_P}^2= (u-x_P)^2+v^2-{r_P}^2=\\
    =(u^2+v^2)+(x_P^2-r_P^2)-2 u x_P= \|O\|^2+\P_T-2 u x_P
\end{gather*}

Using this identity with $P=X_3$, $P=X_5$, and $P=X$, we get
\begin{gather*}
    t\P(O,\C_3)+(1-t)\P(O,\C_5)=\\
    =t(\|O\|^2+\P_T-2 u x_3)+(1-t)(\|O\|^2+\P_T-2 u x_5)=\\
    \|O\|^2+\P_T-2 u (t x_3+(1-t)x_5)=\P(O,\C)
\end{gather*}
\end{proof}

An illustration to Proposition \ref{prop:coaxial} is provided by a well-known group of circles coaxial with the circumcircle and the Euler circle, listed in  Table~\ref{tab:coaxal-circles} and shown in Figure~\ref{fig:coaxal-power}.

\begin{table}
    \centering
    \begin{tabular}{|r|c|c|}
    \hline
    name & center & squared radius \\ \hline
        Circumcircle & $X_5$ & $R^2$ \\
        Euler's circle & $X_3$ & $R^2/4$ \\
        Steiner orthoptic circle & $X_2$ & ${\sum{s_i}^2}/18$ \\
        Orthocentroidal & $X_{381}$ & $R^2-\sum{s_i}^2/9$ \\
        Polar$^\dagger$ circle & $X_4$ & $4 R^2-(\sum{s_i}^2)/2$ \\
        Tangential$^\ddagger$ circle& $X_{26}$ & ${R^2}/({16|\prod{\cos\theta_i}|^2})$ \\
        \hline
    \end{tabular}
    \caption{Six coaxial circles with centers on the Euler line \cite[Coaxal System]{mw}. $^\dagger$The squared radius of the polar circle is positive (resp. negative) for obtuse (resp. acute) triangles. Its signed squared radius is used when computing circle power. $^\ddagger$The tangential circle is the only in the list whose center is not a fixed linear combination of $X_3$ and $X_5$.}
    \label{tab:coaxal-circles}
\end{table}

\begin{observation}
The power of the center wrt to the coaxial circles listed on Table~\ref{tab:coaxal-circles} is invariant, with the exception of the tangential circle, since the latter's center is not a fixed linear combination of $X_3$ and $X_5$.
\label{obs:coaxial}
\end{observation}

\begin{figure}
    \centering
    \includegraphics[width=\textwidth]{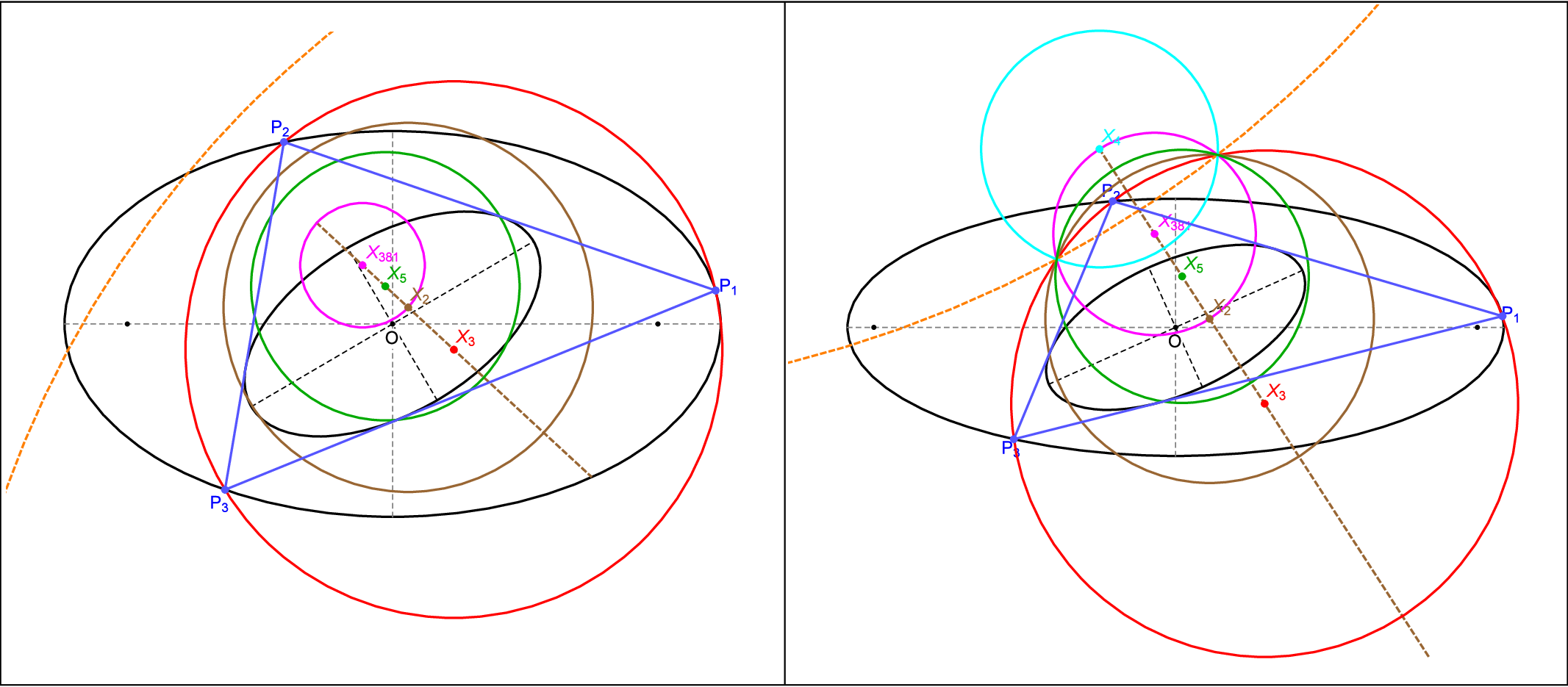}
    \caption{\textbf{Top:} An acute 3-periodic (blue) is shown in a concentric, unaligned pair. The power of the center is invariant with respect to the following set of coaxial circles whose centers lie at linear combinations of $X_3$ and $X_5$: circumcircle (red), Euler's (green), Steiner's inellipse orthoptic (brown), orthocentroidal (magenta), centered on $X_k$, $k=$3,5,2,381, respectively. The power of the center wrt to the also coaxial tangential circle (dashed orange) is variable as expected, since its center $X_{26}$ is not a fixed linear combination of $X_3$ and $X_5$. \textbf{Bottom:} an obtuse 3-periodic (blue) is shown in a concentric, unaligned pair. All coaxial circles now intersect at two common points. The polar circle (cyan) is now defined, centered on $X_4$. The power of the center wrt to all circles shown (except for the tangential, dashed orange) is constant.}
    \label{fig:coaxal-power}
\end{figure}

\begin{proposition}
Over 3-periodics in the concentric, non-axis aligned ellipse pair, the loci of $X_2$ and $X_4$ are ellipses concentric and axis-aligned with the outer one, given by:

\begin{align*}
X_2:& \; 9b^2x^2 + 9a^2y^2 + ab(4a_1b_1 - ab)=0\\
X_4:&\; (a^4 - b^4)a_1^2 - 2ab(a^2 + b^2)a_1b_1 + a^2b^4 - a^2(a^2x^2 + b^2y^2)=0
\end{align*}
\end{proposition}

\subsection*{Generalizing elliptic loci}

The following is a special case of Theorem~\ref{thm:ellipse-locus}, whose proof appears in Section~\ref{sec:nonconcentric-tilted}:

\begin{proposition}
If a triangle center $\X_\gamma=(1-\gamma) X_2+ \gamma X_3$ is a fixed affine combination of $X_2$ and $X_3$ for some $\gamma\in\R$, its locus over 3-periodics in the concentric, non-axis-aligned ellipse pair will be an ellipse.
\label{prop:concentric}
\end{proposition}

There are 226 triangle centers on \cite{etc} which are fixed linear combinations of $X_2$ and $X_3$; see Observation~\ref{obs:affine-euler-line}. Experimentally, these are also the only ones which trace out ellipses.

Consider the converse of Proposition~\ref{prop:concentric}. If for some concentric, non-axis aligned pair some triangle center $\X$ has an elliptic locus, can it be affirmed that $\X$ is a fixed linear combination of $X_2$ and $X_3$? Consider the case of the Steiner point $X_{99}$. This point is known to lie on the Euler line, circumcircle, and Steiner circumellipse, although it is not a fixed linear combination of $X_2$ and $X_3$ \cite{etc}. Still, over (i) the homothetic family, and (ii) the family with circumcircle, its locus is (i) the outer (Steiner) ellipse, and (ii) the outer circle. However, over the confocal family, the locus of $X_{99}$ is non-elliptic.

Based on experimental evidence, in Conjecture~\ref{conj:concentric} (Section~\ref{sec:open-questions}) we state a converse to Proposition~\ref{prop:concentric}.

\section{Non-Concentric with Circumcircle}
\label{sec:nonconcentric-circumcircle}
Here we consider 3-periodics inscribed in a circle and circumscribing a non-concentric ellipse. We will work in the complex plane and apply Blaschke Product techniques \cite{daepp-2019} which simplify our parametrization. Namely, 3-periodic vertices become symmetric with respect to the information of the circle-ellipse pair.

As a first step, identify points in $\R^2$ with points in the complex plane $\Cp$. Let $\D$ denote the open unit disk $\{z\in\Cp : |z|<1\}$ and $\T$ denote the unit circle $\{z\in\Cp : |z|=1\}$. By translation and scaling, we may assume the outer circle of the pair to be the unit circle $\T$. Let $\{f,g\}$ be the two foci of the inner ellipse. As in \cite{daepp-2019}, define:

\begin{definition}{Degree-3 Blaschke Product} 

\[
B(z):=z\left( \frac{z-f}{1-\ol f z}\right) \left( \frac{z-g}{1-\ol g z}\right)
\]

\end{definition}

\noindent Note that if one wants to study the concentric setting, just substitute $g=-f$.

Following chapter 4 of \cite{daepp-2019}, for each $\l\in\T$, the three solutions of $B(z)=\l$ are the vertices of a 3-periodic orbit of the Poncelet family of triangles in the complex plane, and as $\l$ varies in $\T$, the whole family of triangles is covered. Clearing the denominator in this equation and passing everything to the left-hand side, we get

\[
z^3-(f+g+\l\ol{f} \ol{g})z^2+(f g+\l(\ol f+\ol g))z-\l=0
\]

Let $z_1,z_2,z_3\in\Cp$ denote the vertices of Poncelet 3-periodics in the pair with   circumcircle. Using Viète's formula, we obtain the following parametrization of the elementary symmetric polynomials on $z_1,z_2,z_3$:

\begin{definition}[Blaschke's Parametrization]
\begin{align*}
    \sigma_1:=z_1+z_2+z_3=& f+g+\l\ol f \ol g \\
    \sigma_2:=z_1 z_2+z_2 z_3+z_3 z_1=& f g+\l(\ol f+\ol g) \\
    \sigma_3:=z_1 z_2 z_3=& \l
\end{align*}
where $f,g$ are the foci of the inner ellipse and $\l\in\T$ is the varying parameter.
\label{def:bla}
\end{definition}

Referring to Figure~\ref{fig:nonconcentric-circumcircle-circular-loci-right-tris}:

\begin{proposition}
If a triangle center $\X\ab=\alpha X_2+ \beta X_3$ is a fixed linear combination of $X_2$ and $X_3$ for some $\alpha,\beta\in\mathbb{C}$, its locus over 3-periodics in the non-concentric pair with a circumcircle is a circle centered on $\mathcal{O}_\alpha$ and of radius $\mathcal{R}_\alpha$ given by:

\[ \mathcal{O}_\alpha = \frac{\alpha(f+g)}{3},\;\;\; \mathcal{R}_\alpha =\frac{|\alpha f g|}{3}\]
\label{prop:LinComb-concentric}
\end{proposition}

\begin{observation}
Notice that the center and radius of the locus do not depend on $\beta$ since the circumcenter $X_3$ is stationary at the origin of this system.
\end{observation}

\begin{proof}
Since, $z_1,z_2,z_3$ are the 3 vertices of the Poncelet triangle inscribed in the unit circle, its barycenter and circumcenter are given by $X_2=(z_1+z_2+z_3)/3$ and $X_3=0$, respectively. We define $\X\ab:=\alpha X_2+ \beta X_3=\alpha (z_1+z_2+z_3)/3$. Using Definition~\ref{def:bla}, we get $\X\ab=\alpha(f+g+\l \ol{f}\ol{g})/3=\alpha(f+g)/3+\l(\alpha \ol{f}\ol{g})/3$, where the parameter $\l$ varies on the unit circle $\T$. Thus, the locus of $\X_{\gamma}$ over the Poncelet family of triangles is a circle with center $\mathcal{O}_{\alpha}:=\alpha(f+g)/3$ and radius $\mathcal{R}_{\alpha}:=|\alpha \ol{f}\ol{g}|/3=|\alpha f g|/3$.
\end{proof}

Using $\alpha=1-\gamma, \beta=\gamma$ for a fixed $\gamma\in\R$ in Proposition \ref{prop:LinComb-concentric}, we get:

\begin{corollary}
 If a triangle center $\X_\gamma=(1-\gamma) X_2+ \gamma X_3$ is a real affine combination of $X_2$ and $X_3$ for some $\gamma\in\R$, its locus over 3-periodics in the non-concentric pair with a circumcircle is a circle. Moreover, as we vary $\gamma$, the centers of these loci are collinear with the fixed circumcenter.
 \label{cor:gamma-with-circumcircle}
\end{corollary}

Many triangle centers in \cite{etc} are affine combinations of the barycenter $X_2$ and circumcenter $X_3$. See Observation \ref{obs:affine-euler-line} for a compilation of them.

\begin{observation}
For a generic triangle, only $X_{98}$, and $X_{99}$ are simultaneously on the Euler line and on the circumcircle. However these are not linear combinations of $X_2$ and $X_3$. Still, if a triangle center is always on the circumcircle of a generic triangle (there are many of these, see \cite[Circumcircle]{mw}), its locus over 3-periodics in the non-concentric pair with circumcircle is trivially a circle.
\end{observation}

\begin{corollary}
 Over the family of 3-periodics inscribed in a circle and circumscribing a non-concentric inellipse centered at $O_c$, the locus of $X_k$, $k$ in 2,4,5,20 are circles whose centers are collinear. The locus of $X_5$ is centered on $O_c$. The centers and radii of these circular loci are given by:

\begin{alignat*}{4}
    O_2&=\frac{f+g}{3},\quad& O_4&=f+g,\quad&O_5&=\frac{f+g}{2},\quad&O_{20}&=-(f+g)\\
    r_2&=\frac{|f g|}{3},\quad&r_4 &= |f g|,\quad&r_5 &= \frac{|f g|}{2},\quad& r_{20}&= |f g|
\end{alignat*}

\end{corollary}

\begin{proof}
As in Corollary \ref{cor:gamma-with-circumcircle}, we can use Proposition \ref{prop:LinComb-concentric} with $\gamma=0,-2,-1/2,4$ to get the center and radius for $X_2,X_4,X_5,X_{20}$, respectively. All of these centers are real multiples of $f+g$, so they are all collinear. Moreover, the center $O_5$ of the circular loci of $X_5$ is $(f+g)/2$, that is, the midpoint of the foci of the inellipse, or in other words, the center $O_c$ of the inellipse.
\end{proof}
 
Referring to Figure~\ref{fig:nonconcentric-circumcircle-circular-loci-right-tris}:

\begin{observation}
The family of 3-periodics in the pair with circumcircle includes obtuse triangles if and only if $X_3$ is exterior to the caustic. \end{observation}

This is due to the fact that when $X_3$ is interior to the caustic, said triangle center can never be exterior to the 3-periodic. Conversely, if $X_3$ is exterior, it must also be external to some 3-periodic, rendering the latter obtuse.

\begin{figure}
    \centering
    \includegraphics[width=\textwidth]{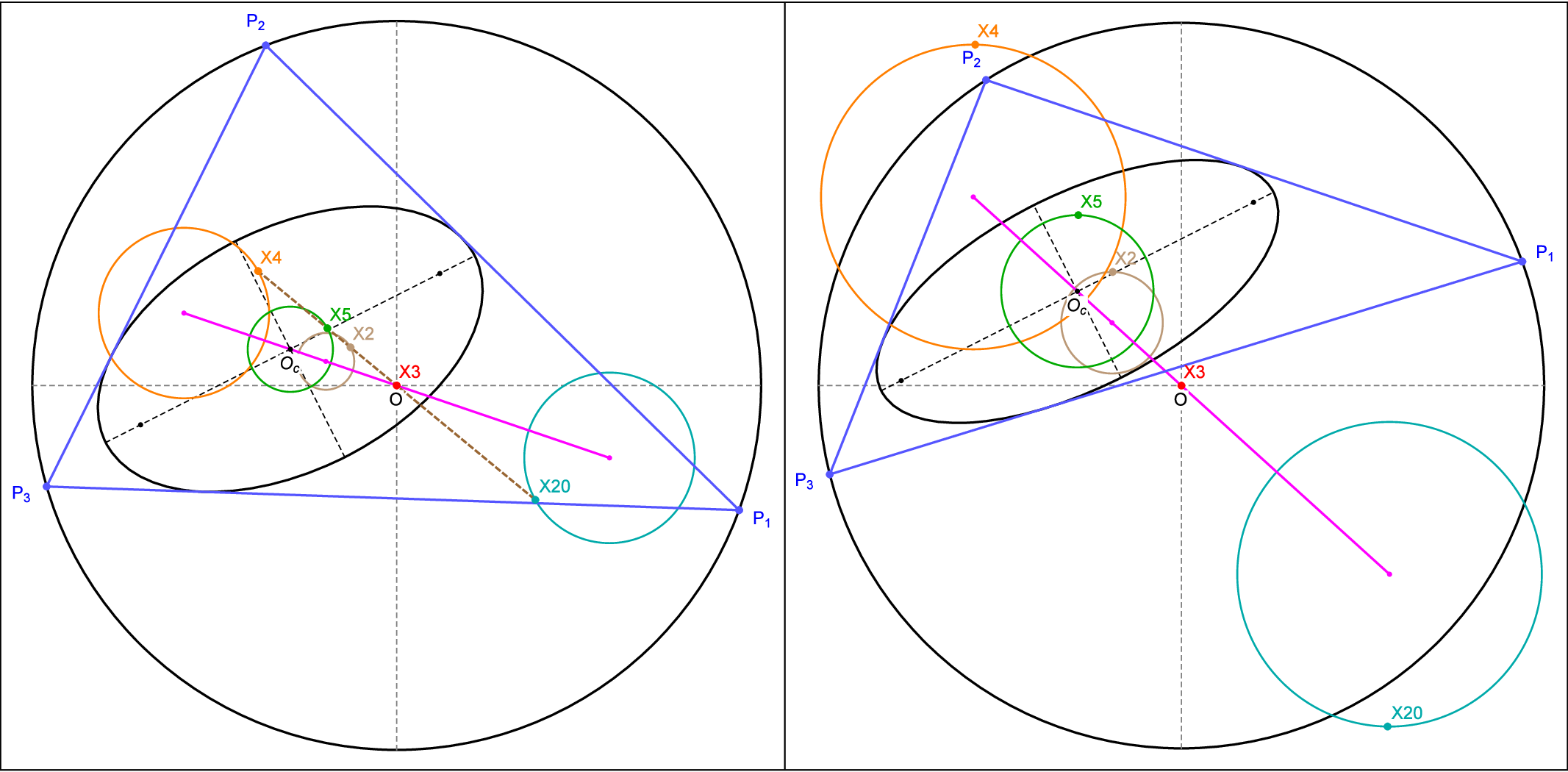}
    \caption{\textbf{Left:} 3-periodic family (blue) in the pair with circumcircle where the caustic contains $X_3$, i.e., all 3-periodics are acute. The loci of $X_4$ and $X_{20}$ are interior to the circumcircle. \textbf{Right:} $X_3$ is exterior to the caustic, and 3-periodics can be either acute or obtuse. Equivalently, the locus of $X_4$ intersects the circumcircle. In both cases (left and right), the loci of $X_k$, $k$ in 2,4,5,20 are circles with collinear centers (magenta line). The locus of $X_5$ is centered on $O_c$. The center of the $X_2$ locus is at $2/3$ along $O O_c$. \href{https://youtu.be/HXgJQo2UT_8}{Video}}
    \label{fig:nonconcentric-circumcircle-circular-loci-right-tris}
\end{figure}

\section{Generic Nested Ellipses}
\label{sec:nonconcentric-tilted}
In this Section we prove the locus of a given fixed linear combination of $X_2$ and $X_3$ is an ellipse. We will continue to use Blaschke product techniques since a generic non-concentric pair can always be seen as the affine image of a pair with circumcircle.

Consider the generic pair of nested ellipses $\E=(O,a,b)$ and $\E_c=(O_c,a_c,b_c.\theta)$ in Figure~\ref{fig:n3-general-pos}. Let s$\theta$, c$\theta$ denote the sine and cosine of $\theta$, respectively. Define $c_c^2=a_c^2-b_c^2$. The Cayley condition for the pair to admit a 3-periodic family is given by:

{\small
\begin{align}
&{b}^{4}x_c^{4}+2\,{a}^{2}{b}^{2}x_c^{2}y_c^{2}+
 \left(  2 c_c^2  \left( -{b}^{2}({a}^{2}+{b}^{2} )\right)  \text{c}\theta^2  - 2\left(  \,b ^{2}- \,b_c
^{2} \right) {b}^{2}{a}^{2}-2\,{b}^{4}b_c^{2} \right)x_c
^{2} \label{eqn:cayley}\\
&-8\,{a}^{2}{b}^{2}x_c\,{  y_c}\,c_c^2 
\text{s}\theta\text{c}\theta  +{a}^{4}y_c^{4} + \left(  2 c_c^2 a^2 \left(
{a}^{2}+{b}^{2}  \right)\text{c}\theta^2  
 -2 \left(  \,b_c^{2}+{b}^{2} \right) {a}^{4}+2
\,{a}^{2}{b}^{2}b_c^{2} \right) y_c^{2} \nonumber\\
&+ c_c^4  c^4  \left( \text{c}\theta^4-2\, c_c^2 c^2
  \left( {a}^{2} a_c^{2}-{b}^{2}{a}^{2}+
b_c^{2}{b}^{2} \right) \text{c}\theta^2 \right. \nonumber\\
 &+ \left( a a_c+a b-b b_c \right)  \left( a a_c
-a b -b b_c \right)  \left( a a_c+a b+b b_c \right)  \left( a
a_c-a b+b b_c \right) = 0\nonumber
\end{align}
}

 
 

Before moving on, we first prove a small parametrization lemma for complex coordinates:

\begin{lemma}
If $u,v,w\in\mathbb{C}$ and $\l$ is a parameter that varies over the unit circle $\T\subset\mathbb{C}$, then the curve parametrized by
\[ F(\l)=u \l+v\frac{1}{\l}+w \]
is an ellipse centered at $w$, with semiaxis $|u|+|v|$ and $\big||u|-|v|\big|$, rotated with respect to the canonical axis of $\mathbb{C}$ by an angle of $(\arg u+\arg v)/2$.
\label{lem:ell-param}
\end{lemma}

\begin{proof}
If either $u=0$ or $v=0$, the curve $h(\T)$ is clearly the translation of a multiple of the unit circle $\T$, and the result follows. Thus, we may assume $u\neq 0$ and $v\neq 0$.

Choose $k\in\mathbb{C}$ such that $k^2=u/v$. Write $k$ in polar form, as $k=r \mu$, where $r>0$ ($r\in\R$) and $|\mu|=1$. We define the following complex-valued functions:
\[R(z):=\mu z,~ S(z):=r z+(1/r) \ol{z},~ H(z):=k v z,~ T(z):=z+w\]

One can straight-forwardly check that $F=T\circ H\circ S\circ R$.

Since $|\mu|=1$, $R$ is a rotation of the plane, thus $R$ sends the unit circle $\T$ to itself. Since $r\in\R$, $r>0$, if we identify $\mathbb{C}$ with $\R^2$, $S$ can be seen as a linear transformation that sends $(x,y)\mapsto\left(\left(r+1/r\right)x,\left(r-1/r\right)y\right)$. Thus, $S$ sends $\T$ to an axis-aligned, origin-centered ellipse $\E_1$ with semiaxis $r+1/r$ and $|r-1/r|$. $H$ is the composition of a rotation and a homothety. $H$ sends the ellipse $\E_1$ to an origin-centered ellipse $\E_2$ rotated by an angle of $\arg(k v)=\arg(k)+\arg(v)=(\arg(u)-\arg(v))/2+\arg(v)=(\arg(u)+\arg(v))/2$. The semiaxis of $\E_2$ have length
\begin{align*}
|k v|&(r+1/r)=r|v|(r+1/r)=|r^2 v|+|v|=|k^2 v|+|v|=|u|+|v|\text{, and}\\
|k v|&|r-1/r|=r|v||r-1/r|=\big||r^2 v|-|v|\big|=\big||k^2 v|-|v|\big|=\big||u|-|v|\big|
\end{align*}

Finally, $T$ is a translation, thus $T$ sends $\E_2$ to an ellipse $\E_3$ centered at $w$, rotated by an angle $(\arg(u)+\arg(v))/2$ from the axis, with semiaxis lengths $|u|+|v|$ and $\big||u|-|v|\big|$, as desired.
\end{proof}

Recall that over Poncelet N-periodics interscribed in a generic pair of conics, the locus of vertex and area centroids is an ellipse \cite{sergei2016-com} as is that of the circumcenter-of-mass \cite{sergei2014-circumcenter-of-mass}, a generalization of $X_3$ for $N>3$. Referring to Figure~\ref{fig:nonconcentric-xns}:

\begin{theorem}
Over the family of 3-periodics interscribed in an ellipse pair in general position (non-concentric, non-axis-aligned),
if $\X\ab$ is a fixed linear combination of $X_2$ and $X_3$, i.e., $\X\ab=\alpha X_2+\beta X_3$ for some fixed $\alpha,\beta\in\mathbb{C}$, then its locus is an ellipse. 
\label{thm:ellipse-locus}
\end{theorem}

\begin{proof}

Consider a general $N=3$ Poncelet pair of ellipses that forms a 1-parameter family of triangles. Without loss of generality, by translation and rotation, we may assume the outer ellipse is centered at the origin and axis-aligned with the plane $\R^2$, which we will also identify with the complex plane $\mathbb{C}$. Let $a,b$ be the semi-axis of the outer ellipse, and $a_c,b_c$ the semi-axis of the inner ellipse, as usual. 

Referring to Figure~\ref{fig:affine}, consider the linear transformation that takes $(x,y)\mapsto(x/a,y/b)$. This transformation takes the outer ellipse to the unit circle $\T$ and the inner ellipse to another ellipse. Thus, it transforms the general Poncelet $N=3$ system into a pair where the outer ellipse is the circumcircle, which we can parametrize using Blaschke products \cite{daepp-2019}. In fact, to get back to the original system, we must apply the inverse transformation that takes $(x,y)\mapsto(a x,b y)$. As a linear transformation from $\mathbb{C}$ to $\mathbb{C}$, we can write it as $L(z):=p z+q \ol{z}$, where $p:=(a+b)/2, q:=(a-b)/2$.

\begin{figure}
    \centering
    \includegraphics[width=\textwidth]{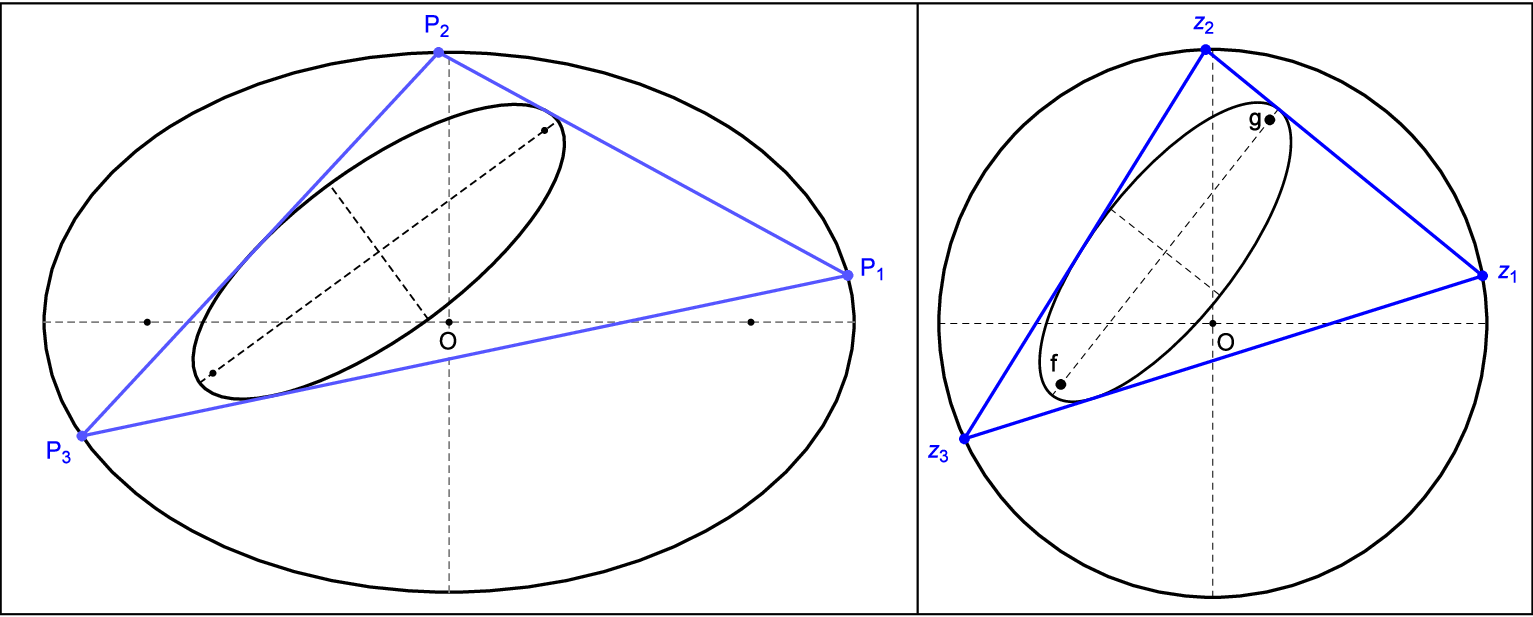}
    \caption{Affine transformation that sends a generic ellipse pair and its 3-periodic family (left) to a new pair with circumcircle (right). We parametrize the 3-periodic orbit with vertices $z_i$ in the circumcircle pair using the foci of the latter's caustic $f$ and $g$, and then apply the inverse affine transformation to get a parametrization of the vertices $P_i$ of the original Poncelet pair. \href{https://youtu.be/6xSFBLWIkTM}{Video}}
    \label{fig:affine}
\end{figure}

Let $z_1,z_2,z_3\in\T\subset\mathbb{C}$ be the three vertices of the circumcircle family, parametrized as in Definition~\ref{def:bla}, and let $v_1:=L(z_1),v_2:=L(z_2),v_3:=L(z_3)$ be the three vertices of the original general family. The barycenter $X_2$ of the original family is given by $(v_1+v_2+v_3)/3$, and the circumcenter $X_3$ is given by \cite{stackexchange-x3a}:

\[
    X_3=\left|
        \begin{array}{ccc}
          v_1 & |v_1|^2 & 1 \\
          v_2 & |v_2|^2 & 1 \\
          v_3 & |v_3|^2 & 1
        \end{array}
      \right| \Bigg/
     \left|
        \begin{array}{ccc}
          v_1 & \overline{v_1} & 1 \\
          v_2 & \overline{v_2} & 1 \\
          v_3 & \overline{v_3} & 1
        \end{array}
      \right|
\]

Since $\ol{z_1}=1/z_1,\ol{z_2}=1/z_2,\ol{z_3}=1/z_3$, we can write $v_1,v_2,v_3$ as rational functions of $z_1,z_2,z_3$, respectively. Thus, both $X_2$ and $X_3$ are symmetric rational functions on $z_1,z_2,z_3$. Defining $\X\ab=\alpha X_2+\beta X_3$, we have consequently that $\X\ab$ is also a symmetric rational function on $z_1,z_2,z_3$. Hence, we can reduce its numerator and denominator to functions on the elementary symmetric polynomials on $z_1,z_2,z_3$. This is exactly what we need in order to use the parametrization by Blaschke products.

In fact, we explicitly compute:
\[  \X\ab= \frac{p^2 q \left(\sigma_2 (\alpha +3 \beta )+3 \beta  \sigma_3^2\right)+\alpha  p^3 \sigma_1 \sigma_3-p q^2 (3 \beta +\sigma_1 \sigma_3 (\alpha +3 \beta ))-\alpha  q^3 \sigma_2}{3 \sigma_3 (p-q) (p+q)}\]
where $\sigma_1,\sigma_2,\sigma_3$ are the elementary symmetric polynomials on $z_1,z_2,z_3$.

Let $f,g\in\mathbb{C}$ be the foci of the inner ellipse in the circumcircle system. Using Definition~\ref{def:bla}, with the parameter $\l$ varying on the unit circle $\T$, we get:

\begin{equation}
\X\ab= u \l+v\frac{1}{\l}+w
\label{eqn:xi-param}
\end{equation}

\noindent where:

\begin{align*}
    u:=&\frac{p \left(\ol{f} \ol{g} \left(\alpha  p^2-q^2 (\alpha +3 \beta )\right)+3 \beta  p q\right)}{3 (p-q) (p+q)}\\
    v:=&\frac{\beta  p q (q-f g p)}{(q-p) (p+q)}+\frac{1}{3} \alpha  f g q\\
    w:=&\frac{q \left(\ol{f}+\ol{g}\right) \left(p^2 (\alpha +3 \beta )-\alpha  q^2\right)+p (f+g) \left(\alpha  p^2-q^2 (\alpha +3 \beta )\right)}{3 (p-q) (p+q)}
\end{align*}

By Lemma \ref{lem:ell-param}, this is the parametrization of an ellipse centered at $w$, as desired. As in Lemma \ref{lem:ell-param}, it is also possible to explicitly calculate its axis and rotation angle, but these expressions become very long.

\end{proof}

\begin{corollary}
Over the family of 3-periodics interscribed in an ellipse pair in general position (non-concentric, non-axis-aligned),
if $\X_\gamma$ is a real affine combination of $X_2$ and $X_3$, i.e., $\X_\gamma=(1-\gamma) X_2+\gamma X_3$ for some fixed $\gamma\in\R$, then its locus is an ellipse. Moreover, as we vary $\gamma$, the centers of the loci of the $\X_\gamma$ are collinear.
\end{corollary}

\begin{proof}
Apply Theorem \ref{thm:ellipse-locus} with $\alpha=1-\gamma, \beta=\gamma$ to get the elliptical loci. As in the end of the proof of Theorem \ref{thm:ellipse-locus}, the center of the locus of $\X_\gamma$ can be computed explicitly as 
\begin{gather*}
    w=w_0+w_1 \gamma \text{, where}\\
    w_0=\frac{1}{3} \left(q \left(\ol{f}+\ol{g}\right)+p (f+g)\right)\\
    w_1=\frac{q \left(2 p^2+q^2\right) \left(\ol{f}+\ol{g}\right)-p (f+g) \left(p^2+2 q^2\right)}{3 (p-q) (p+q)}
\end{gather*}
As $\gamma\in\R$ varies, it is clear the center $w$ sweeps a line.
\end{proof}

We proved that all of the following triangle centers have elliptic loci in the general N=3 Poncelet system, including the barycenter, circumcenter, orthocenter, nine-point center, and de Longchamps point:

\begin{observation}
Amongst the 40k+ centers listed on \cite{etc}, about 4.9k triangle centers lie on the Euler line \cite{etc-central-lines}. Out of these, only 226 are fixed affine combinations of $X_2$ and $X_3$. For $k<1000$, these amount to $X_k,k=${\small 2, 3, 4, 5, 20, 140, 376, 381, 382, 546, 547, 548, 549, 550, 631, 
632}.
\label{obs:affine-euler-line}
\end{observation}

\begin{figure}
     \centering
     \includegraphics[width=.8\textwidth]{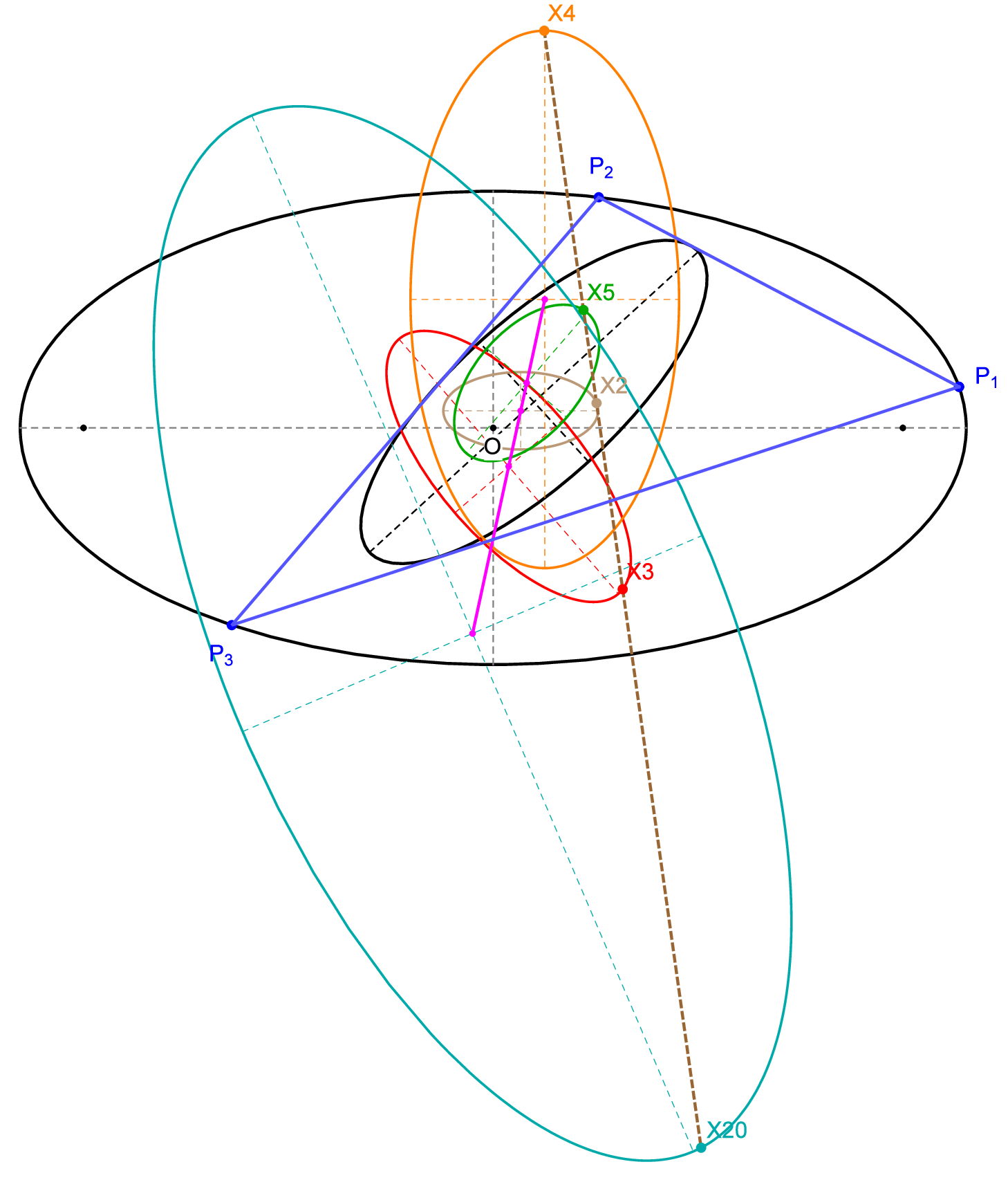}
     \caption{A 3-periodic is shown interscribed between two nonconcentric, non-aligned ellipses (black). The loci of $X_k$, $k=2,3,4,5,20$ (and many others) remain ellipses. Those of $X_2$ and $X_4$ remain axis-aligned with the outer one. Furthermore the centers of all said elliptic loci are collinear (magenta line). \href{https://youtu.be/p1medAei_As}{Video}}
     \label{fig:nonconcentric-xns}
 \end{figure}
 
 
\begin{observation}
The elliptic loci of $X_2$ and $X_4$ are axis-aligned with the outer ellipse.
\end{observation} 

We conclude this section with phenomenon specific to the case where $\E_c$ is a circle, Figure~\ref{fig:circular-caustic}:

\begin{observation}
 Over the family of 3-periodics inscribed in an ellipse and circumscribing a non-concentric circle centered on $O_c=X_1$, the locus of $X_3$ and $X_5$ are ellipses whose major axes pass through $X_1$.
 \end{observation}
 
\begin{figure}
    \centering
    \includegraphics[width=.7\textwidth]{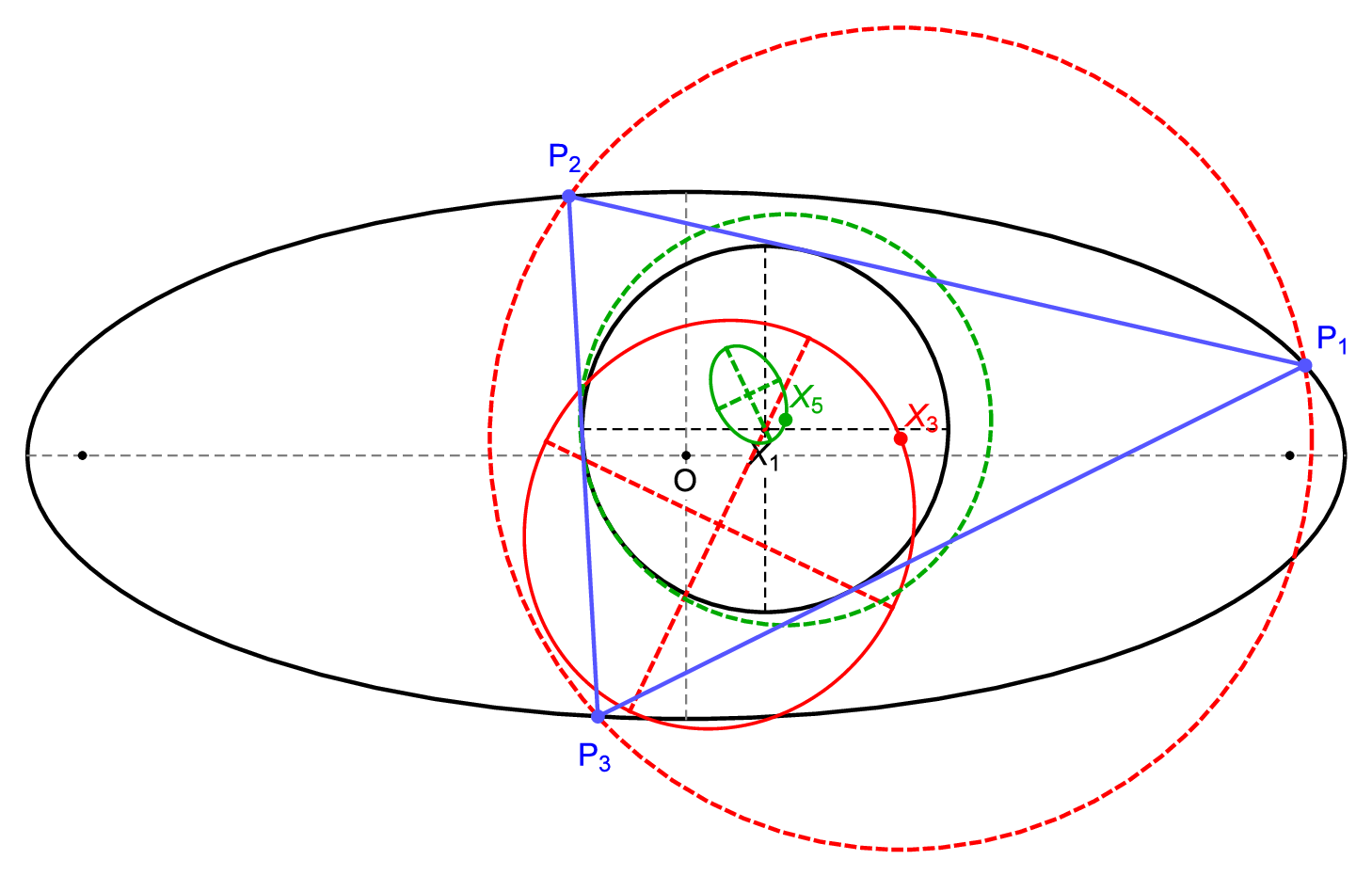}
    \caption{A 3-periodic (blue) is shown inscribed in an outer ellipse and an inner non-concentric circle centered on $O_c$. The loci of both circumcenter (solid red) and Euler center (solid green) are ellipses whose major axes pass through $O_c$. \href{https://youtu.be/w7sZ5O8k4xU}{Video}}
    \label{fig:circular-caustic}
\end{figure}

\section{Conjectures and Videos}
\label{sec:open-questions}
Referring to Figure~\ref{fig:p3p5-non-concentric}:

\begin{conjecture}
Over 3-periodics in the non-concentric, non-axis-aligned pair, there is some fixed point $P_3$ (resp. $P_5$) such that its power with respect to the circumcircle (resp. Euler circle) is invariant. 
\end{conjecture}

\begin{figure}
    \centering
    \includegraphics[width=.6\textwidth]{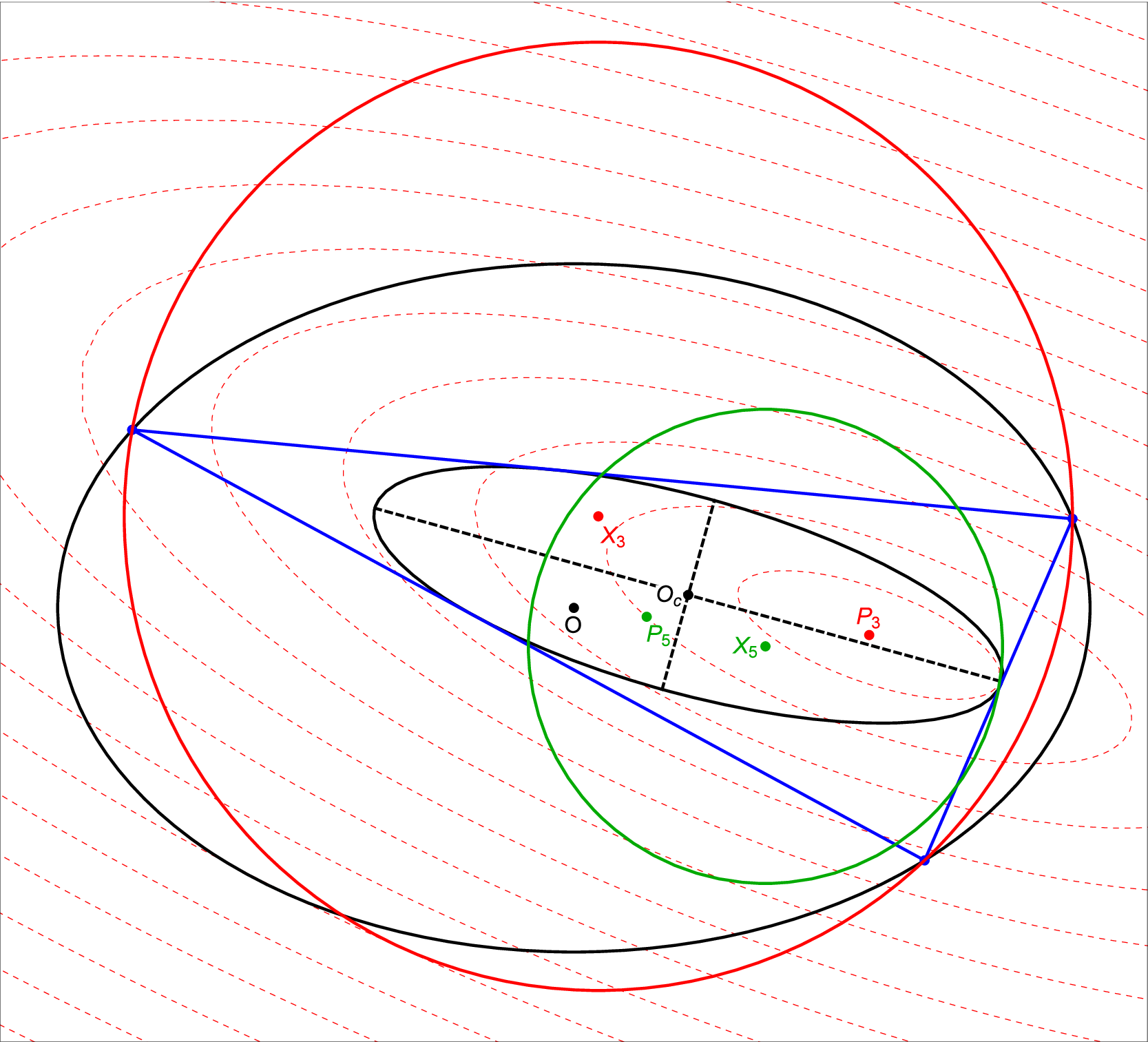}
    \caption{Consider a 3-periodics (blue) in a pair of ellipses in general position (centers at $O$ and $O_c$), as well as its circumcircle (red, center $X_3$) and Euler's circle (green, center $X_5$). A point $P_3$ can be numerically located whose power to the circumcircle (solid red) is invariant over the 3-periodic family. Iso-curves of the variance of power of $(x,y)$ with respect to the circumcircle are shown (dashed red): the minimum (and zero) variance occurs at $P_3$. An analogous numeric approach is used to locate the point $P_5$ whose power wrt Euler's circle (solid green) is invariant.}
    \label{fig:p3p5-non-concentric}
\end{figure}

In \cite{olga14,garcia2019-incenter} it is shown that in the confocal pair the locus of the incenter $X_1$ is an ellipse.

Experimentally, we can strengthen Proposition~\ref{prop:concentric}: 

\begin{conjecture}
If the locus of a triangle center $\X$ is an ellipse for all concentric, non-axis-aligned ellipse pairs, then $\X$ is a fixed linear combination of $X_2$ and $X_3$. 
\label{conj:concentric}
\end{conjecture} 

Our very first experimental result was that over 3-periodics in the confocal pair, the locus of the incenter $X_1$ was an ellipse \cite{reznik2011-incenter}. This was subsequently proved \cite{olga14,garcia2019-incenter}. Considering the space of choices of nested ellipse pairs is 5d (5 parameters for each minus 4d homothethy group, minus 1d Cayley condition), little did we know how rare a phenomenon that was (the space of confocal ellipses is a 1d): 

\begin{conjecture}
 The only pair of ellipses admitting Poncelet 3-periodics such that the locus of the incenter $X_1$ is an ellipse is the confocal pair.
\end{conjecture}

\subsection*{Videos} Animations illustrating some phenomena herein are listed on Table~\ref{tab:playlist}.

\begin{table}[H]
\small
\begin{tabular}{|c|l|l|}
\hline
id & Title & \textbf{youtu.be/<.>}\\
\hline
01 & {Cayley-Poncelet Phenomena I: Basics} &
\href{https://youtu.be/virCpDtEvJU}{\texttt{virCpDtEvJU}}\\
02 & {Cayley-Poncelet Phenomena II: Intermediate} &
\href{https://youtu.be/4xsm\_hQU-dE}{\texttt{4xsm\_hQU-dE}}\\
03 & {3-Periodics in Non-Concentric, Unaligned Pair} &
\href{https://youtu.be/bjHpXVyXXVc}{\texttt{bjHpXVyXXVc}}\\
04 & {Loci of Centers I: Generic Pair} &
\href{https://youtu.be/p1medAei_As}{\texttt{p1medAei\_As}}\\
05 & {Loci of Centers II: Pair with Circumcircle} &
\href{https://youtu.be/HXgJQo2UT_8}{\texttt{HXgJQo2UT\_8}}\\
06 & {Loci of Centers III: Concentric Tilted Pair} &
\href{https://youtu.be/hpb7ZgKWjUY}{\texttt{hpb7ZgKWjUY}}\\
07 & {Loci of Centers IV: Outer Ellipse, Inner Non-Concentric Circle} &
\href{https://youtu.be/w7sZ5O8k4xU}{\texttt{w7sZ5O8k4xU}}\\
08 & {3-Periodics in Generic Pair + Affine Image w/ Circumcircle} &
\href{https://youtu.be/6xSFBLWIkTM}{\texttt{6xSFBLWIkTM}}\\
\hline
\end{tabular}
\caption{Videos of some focus-inversive phenomena. The last column is clickable and provides the YouTube code.}
\label{tab:playlist}
\end{table}

\section*{Acknowledgements}
\noindent We would like to thank A. Akopyan for valuable insights. The third author is fellow of CNPq and coordinator of Project PRONEX/ CNPq/ FAPEG 2017 10 26 7000 508.

\appendix

\section{Additional Invariant-Power Circles}
\label{app:four-more}
Consider the well-known circles derived from a reference triangle and their radii, listed on Table~\ref{tab:four-more}.

\begin{table}
\begin{tabular}{|l|l|l|l|}
\hline
Circle & Description (see \cite{mw}) & Center & Radius\\
\hline
Anticompl. & Circumc. of Anticompl. & $X_{4}$ & $2R$\\
Bevan & Circumc. of Excentral & $X_{40}$ & $2R$ \\
Spieker & Incircle of Medial & $X_{10}$ & $r/2$ \\
Mandart & Circumc. of Extouch & $X_{1158}$ & see \eqref{eqn:mandart-radius} \\
\hline
\end{tabular}
\caption{Definition of certain circles for which under certain families the power of the center is also invariant. $R,r$ denote circumradius and inradius.}
\label{tab:four-more}
\end{table}

Let $s$ denote the semiperimeter. The radius of the Mandart circle is given by \cite[Mandart Circle]{mw}:

\begin{equation}
R_m=\frac{s}{s_1 s_2 s_3}\sqrt{(4R^2-s_2 s_3)(4 R^2-s_3 s_1)(4R^2-s_1 s_2)}
\label{eqn:mandart-radius}
\end{equation}

Referring to Figure~\ref{fig:four-more}, for selected families, the power of the center is also invariant.

\begin{proposition}
Over the homothetic family, the power of the center with respect to the anticomplementary circle is given by:

\[ \P_{4,act}= {-(4/9)\sum{s_i}^2}\]
\end{proposition}

\begin{proposition}
Over the confocal family, the power of the center with respect to the Bevan, Spieker, and Mandart circles are given by:

\begin{align*}
    \P_{40,bev}=& -(a^2+b^2+2\delta)\\
    \P_{10,spi}=& -(1/64)(1-h^2)^2(a^2+b^2+2\delta)\\
    \P_{1158,man}=& -(1/48)(-h^4+14h^2+3)(a^2+b^2+2\delta)
\end{align*}
with $h = ({-a^2-b^2+2\delta})/c^2$.
\end{proposition}

\begin{figure}
 \centering
 \includegraphics[width=.8\textwidth]{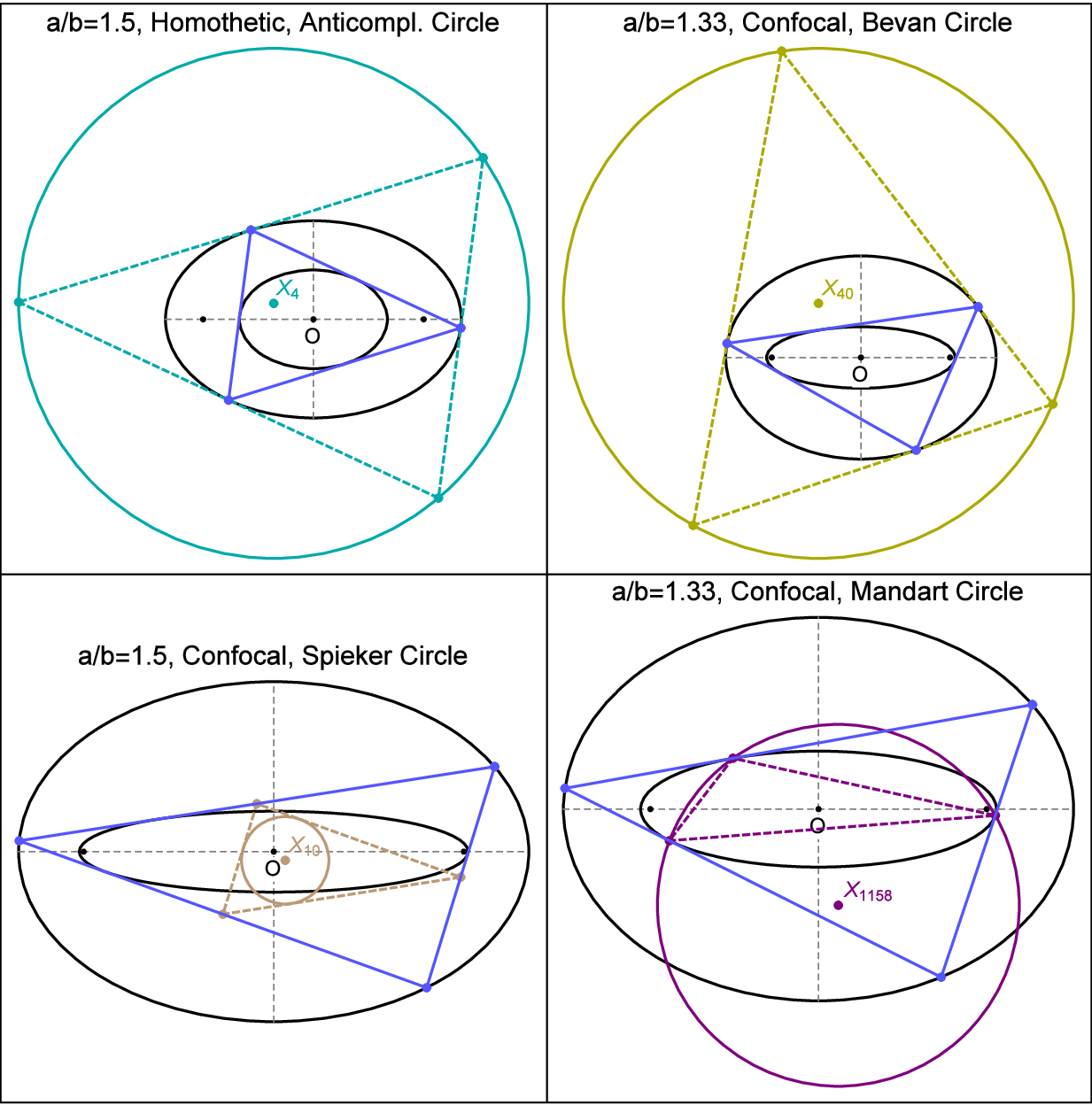}
 \caption{Four additional circles (Anticomplementary, Bevan, Spieker, Mandart) with respect to which the power of the center is invariant, for the particular families indicated (homothetic, and thrice confocal, respectively). See these in motion at: \href{https://bit.ly/3qiTsgY}{\texttt{bit.ly/3qiTsgY}}, \href{https://bit.ly/3jPVvqf}{\texttt{bit.ly/3jPVvqf}}, \href{https://bit.ly/3bhOVFp}{\texttt{bit.ly/3bhOVFp}}, \href{https://bit.ly/3poGBbW}{\texttt{bit.ly/3poGBbW}}, respectively.}
 \label{fig:four-more}
\end{figure}

Additional circles and families show on Table~\ref{tab:addtl-circles} have been detected experimentally, with respect to which the center has constant power. We challenge the reader to derive them.

\begin{table}
\begin{tabular}{|l|l|l|l|}
\hline
Family & Triangle & Circle & Animation \\
\hline
Confocal & Extouch & Euler's & \href{https://bit.ly/3phwBkz}{\texttt{bit.ly/3phwBkz}} \\
Incircle & Intouch & Euler's & \href{https://bit.ly/2ZapHD0}{\texttt{bit.ly/2ZapHD0}} \\
Homothetic & Medial & Euler's & \href{https://bit.ly/3rTNQdc}{\texttt{bit.ly/3rTNQdc}} \\
Circumcircle & Euler's$^\ddagger$ & Euler's & \href{https://bit.ly/3qjytdY}{\texttt{bit.ly/3qjytdY}} \\
Dual$^\dagger$ & Euler's & Euler's & \href{https://bit.ly/2Nlx7AS}{\texttt{bit.ly/2Nlx7AS}}\\
Dual & Anticompl. & Circumc. & \href{https://bit.ly/3ai7BWg}{\texttt{bit.ly/3ai7BWg}}\\
Circumcircle & Orthic & Incircle & \href{https://bit.ly/3qhjhy0}{\texttt{bit.ly/3qhjhy0}} \\
\hline
\end{tabular}
\caption{Experimentally, the power of the center wrt certain additional family-circle combinations is also invariant. $^\ddagger$Euler's Triangle \cite{mw} has vertices at the midpoints of lines from the orthocenter $X_4$ to the vertices (they lie on Euler's circle).}
\label{tab:addtl-circles}
\end{table}

All of the examples in this section can be viewed in motion in \href{https://bit.ly/37dr1JJ}{bit.ly/37dr1JJ}.

\section{Table of Symbols}
\label{app:symbols}
\begin{table}[H]
\small
\begin{tabular}{|c|l|}
\hline
symbol & meaning \\
\hline
$\E,\E_c$ & outer and inner ellipses \\
$O,O_c$ & centers of $\E$,$\E_c$\\
$a,b,a_c,b_c$ & outer and inner ellipse semi-axes' lengths \\
$c,c_c$ & half-focal length of $\E,\E_c$ \\ 
 $O_c=(x_c,y_c)$\\
$\theta$ & major semi-axis tilt $\E_c$ wrt $\E$ \\
$P_i,s_i$ & 3-periodic vertices and sidelengths \\
$r,R$ & 3-periodic inradius and circumradius \\
$a_i,b_i$ & semiaxes of the locus of $X_i$ \\
$r_i$ & radius of the locus of $X_i$ (if $a_i=b_i$) \\
\hline
$\C_3,\C_5$ & circum- and Euler's circle \\
$\C_2,\C_{381}$ & Steiner orthoptic and orthocentroidal circle \\
$\C_4,\C_{26}$ & polar and tangential circle \\
$\P_i$ & power of center $O$ wrt $\C_i$ \\
\hline
$X_1,X_2,X_3$ & Incenter, Barycenter, Circumcenter \\
$X_4,X_5,X_6$ & Orthocenter, Euler center, Symmedian point \\
$X_9,X_{20}$ & Mittenpunkt, de Longchamps point\\
\hline
\end{tabular}
\caption{Symbols used in the article.}
\label{tab:symbols}
\end{table}

\bibliographystyle{maa}
\bibliography{references,authors_rgk_v3}

\end{document}